\documentclass[12pt]{amsart}

\usepackage{latexsym, amssymb, amsmath,ulem,soul,esint}

\usepackage{amsfonts, graphicx}
\usepackage{graphicx,color}
\newcommand{\be}{\begin{equation}}
\newcommand{\ee}{\end{equation}}
\newcommand{\beq}{\begin{eqnarray}}
\newcommand{\eeq}{\end{eqnarray}}

\newtheorem{thm}{Theorem}[section]

\newtheorem{ass}{Assumption}[section]
\newtheorem{lma}{Lemma}[section]
\newtheorem{prop}{Proposition}[section]
\newtheorem{cor}{Corollary}[section]

\theoremstyle{remark}
\newtheorem{rem}{Remark}[section]
\numberwithin{equation}{section}

\def\be{\begin{equation}}
\def\ee{\end{equation}}
\def\bee{\begin{equation*}}
\def\eee{\end{equation*}}
\def\ol{\overline}
\def\lf{\left}
\def\ri{\right}

\def\by{\mathbf{y}}

\def\wn{\wt\nabla}
\def\cR{\mathcal{S}}

\def\Ric{\text{\rm Ric}}

\def\cS{\mathcal{S}}

\def\H{\mathbb{H}}

\def\wt{\widetilde}
\def\la{\langle}
\def\ra{\rangle}
\def\p{\partial}

\def\ol{\overline}

\def\e{\varepsilon}
\def\a{{\alpha}}
\def\b{{\beta}}

\def\R{\mathbb{R}}

\def\mS{\mathbb{S}}

\def\ve{\varepsilon}

\begin{document}

\title[]
{Boundary behaviors of spacelike constant mean curvature surfaces in Schwarzschild spacetime}

\author{Caiyan Li$^1$}
 \address[Caiyan Li]{Key Laboratory of Pure and Applied Mathematics, School of Mathematical Sciences, Peking University, Beijing, 100871, P.\ R.\ China}
\email{lcy@math.pku.edu.cn}
\thanks{$^1$Research partially supported by China Postdoctoral Science Found 8206400077}
 \author{Yuguang Shi$^2$}
 \address [Yuguang Shi] {Key Laboratory of Pure and Applied Mathematics, School of Mathematical Sciences, Peking University, Beijing, 100871, P.\ R.\ China}
\email{ygshi@math.pku.edu.cn}
\thanks{$^2$Research partially supported by National Key R\&D Program of China SQ2020YFA070059 and NSFC 11731001.}

\author{Luen-Fai Tam$^3$}
\address[Luen-Fai Tam]{The Institute of Mathematical Sciences and Department of Mathematics, The Chinese University of Hong Kong, Shatin, Hong Kong, China.}
 \email{lftam@math.cuhk.edu.hk}
\thanks{$^3$Research partially supported by Hong Kong RGC General Research Fund \#CUHK 14301517}

\renewcommand{\subjclassname}{
  \textup{2010} Mathematics Subject Classification}

\date{\today}

\begin{abstract}
We prove that a spacelike spherical symmetric constant mean curvature (SSCMC) surface and a general spacelike constant mean curvature (CMC) surface with certain boundary condition at the future null-infinity  in Schwarzschild spacetime are asymptotically hyperbolic in the sense of Wang \cite{Wang2001} and Chru\'sciel-Herzlich \cite{ChruscielHerzlich} respectively. Near the future null-infinity ($s=0$), we derive that the boundary data of spacelike CMC surfaces can be expressed as  those on $\mathbb{S}^{2}$ up to three order and obtain a compatibility condition for fourth order derivatives near $s=0$. We also show that if the trace free part of the second fundamental forms $\mathring A$ of this spacelike CMC surface decay fast enough then the restriction of its associate function $P$ (for definition, see \eqref{defofp} ) on the null-infinity must be a first eigenfunction of the Laplace on $\mathbb{S}^2$ or constant.  In particular in Minkowski spacetime, a uniqueness result and constructions of spacelike CMC surfaces near $s=0$ are proved. Also, we show that the inner boundary of certain spacelike CMC surfaces are totally geodesic.
\end{abstract}


\maketitle
\markboth{Caiyan Li, Yuguang Shi and Luen-Fai Tam }{Boundary behaviors of spacelike constant mean curvature surfaces}
\section{Introduction}\label{s-intro}

In this work, we want to study the boundary behaviors of spacelike constant mean curvature (CMC) surfaces in the exterior region of Schwarzschild spacetime. We will focus on the behaviors of the surfaces near the future null-infinity and near the black hole region.  Corresponding results should be true for the past null-infinity.

In \cite{LL,LL2}, K-W Lee and Y-I Lee gave a complete description of spacelike spherical symmetric constant mean curvature (SSCMC) surfaces in the Kruskal extension of Schwarzschild spacetime. Some behaviors near null-infinity and the black hole region are also studied. We want to discuss more general spacelike CMC surfaces. In order to study the behaviors near  the future null-infinity $\mathcal{I}^+$, it is more convenient to use
retarded null coordinate $v$ together with $s=r^{-1}$ and $y^1,y^2$ which are coordinates of the standard unit sphere $\mS^2$. Consider a spacelike CMC surface $\Sigma$ with constant mean curvature $H_0>0$ defined on $\mathcal{I}^+$ which is a graph of a function $P=P(y^1,y^2,s)$. {\it We always assume that $P$ is smooth on $s_0\ge s>0$ for some $s_0>0$.} We have the following:
\begin{thm}\label{t-intro-1} Suppose $P$ is at least $C^4$ up to $s=0$. Let $f(y^1,y^2)=P(y^1,y^2,0)$ be the boundary value of $P$ at $\mathcal{I}^+$. We have the following, assuming $H_0=1$:
 \begin{enumerate}
   \item [(i)] $P_s, P_{ss}, P_{sss}$ can be expressed in terms of $f$ and its derivatives as a function in $\mS^2$.
   \item [(ii)] In general $P_{ssss}$ cannot be determined by $f$.
   \item [(iii)] $f$ must satisfy a compatibility condition in order that $P$ is $C^4$ smooth.
       \item[(iv)] If $P$ is smooth up to $s=0$, the all derivatives of $P$ at $s=0$ are determined by $f$ and $P_{ssss}|_{s=0}$.
 \end{enumerate}
 \end{thm}
 The results are true in more general setting that the mean curvature may not be constant. See Theorem \ref{t-smooth-future} for more details. For example, the results in (i)--(iii) above are still true for general spacetime which are close to Schwarzschild spacetime near $\mathcal{I}^+$ in some sense. For a SSCMC surface, one can show that $P$ is constant and is smooth at $s=0$. Hence the constant function will satisfy the compatibility condition mentioned above. In fact, if $f$ is a linear combination of the first four eigenfunctions of the Laplacian of $\mS^2$, then it also satisfies the compatibility condition. These particular boundary values turn  out to be related to the decay rate of the traceless part of the second fundamental form.  Namely, we have the following (assuming $P$ is smooth for simplicity):

 \begin{thm}\label{t-mathringA} Let $\mathring A$ be the traceless part of the spacelike surface $\Sigma$ with induced metric $G$. Then  $|\mathring A|_G=O(s)$.

\begin{enumerate}
  \item[(a)] Moreover, the following statements are equivalent:
  \begin{enumerate}
    \item[(i)] $|\mathring A|_G =O(s^{2})$.
    \item[(ii)] $P$ is the linear combination of the first four eigenfunctions of the Laplacian of $\mS^2$ with respect to the standard metric.
    \item [(iii)] $|\mathring A|_G =O(s^{3})$.
  \end{enumerate}
    They are equivalent to the fact that $|\cS+6|=O(s^6)$, where $\cS$ is the scalar curvature of the surface.
  \item [(b)]
    Furthermore, if $m>0$ then the following two statements are equivalent:
  \begin{enumerate}
    \item[(i)] $|\mathring A|_G =O(s^{4})$.
    \item[(ii)] $P$ is constant and $P_{ssss}=-\frac{9m}2$ at $s=0$.
  \end{enumerate}
\end{enumerate}
\end{thm}
We should remark that (b) is not true in case $m=0$ by considering the graph of $u(\mathbf{x})=(1+|\mathbf{x}+\mathbf{a}|^2)^\frac12$ in the Minkowski spacetime.

 As an application, one can prove that a spacelike CMC surface in the Minkowski spacetime which is the entire graph of a function $u$ so that it is smooth at $\mathcal{I}^+$ and with scalar curvature decays like (iv) in the above, then the graph is a hyperboloid by some translation in the space. Given any smooth function $f$ on $\mS^2$, one can also use the above results to construct a spacelike surface near null-infinity with boundary value $f$ and with constant mean curvature up to certain order.  We also prove that if $P$ is as in (ii) above, then the surface is asymptotically hyperbolic in the sense of Chru\'sciel-Herzlich \cite{ChruscielHerzlich}. Similar results should be  true for spacetime which is close to the Schwarzschild. We plan to investigate this in the future.

 We also investigate the behavior of a spacelike CMC surface $\Sigma$ which is smooth near  $T=X=0$   in the Kruskal extension of Schwarzschild spacetime. For simplicity, we adopt the coordinates $(\eta, y^1,y^2)$ (for its definition, see  \S\ref{ss-equation-2m}) to study this case. By Proposition \ref{smoothnessinnerbdry}, we  know that the $t$-function $u$ of $\Sigma$ is actually smooth near $r=2m$. More interestingly, we observe that the derivatives of $u$ with respect to $\eta$ can be expressed in terms of those of function $u$ restricted on the inner boundary $r=2m$ (see Theorem \ref{t-blackhole} for more details). We also show the inner boundary of $\Sigma$ being totally geodesic, that may have some potential application in study of Penrose inequality for $\Sigma$.

 This paper is organized as follows: in Section 2, we first recall some basic facts on Kruskal extension and conformal compactification of Schwarzschild spacetime, and derive the mean curvature equations under various coordinates; in Section 3, we study the behaviors of the spacelike CMC surfaces near the future null-infinity; in Section 4, we explore behaviors of the spacelike CMC surfaces near the inner boundary $r=2m$.
\vskip .1cm

  {\bf Acknowledgement}:This work is motivated by some questions of Prof. Shing-Tung Yau on constant mean curvature surfaces in the Schwarzschild spacetime.The second author would like to thank Prof.Yau for drawing his attention to this problem. The third author would like to thank Profs. Andrejs Treibergs and Jiaping Wang for some useful discussion.

 \section{Preliminary}\label{s-prelim}

 In this section, we will set up the background for discussion. Namely, in order to study the boundary behavior of a spacelike CMC surface, we will use the Kruskal extension of Schwarzschild spacetime. In order to study the behavior of the surface near null-infinity, we will use the standard conformal compactification. We will only consider the behavior near the future null-infinity. The other case is similar.  To do the analysis, we have to write down the mean curvature equation of the function whose graph is a spacelike CMC surface.

Let us first fix some notations. $\mS^2$ will always denote the standard unit sphere in $\R^3$, that is
$$
\mS^2=\{\mathbf{x}\in \R^3: |\mathbf{x}|=1\}.
$$
The standard metric is denoted by $\sigma=\sigma_{AB}dy^Ady^B$ in local coordinates $y^1, y^2$ for $\mathbf{y}\in\mS^2$. $\wn$ and $\wt\Delta$ are the covariant derivative and the Laplacian on $(\mS^2,\sigma)$ respectively.

For the Minkowski spacetime and the original Schwarzschild spacetime, the usual coordinates are given by $t, \mathbf{x}$. Sometimes we write
$(t, \mathbf{x})$ as $(x^0,x^1,x^2,x^3)$. In terms of polar coordinates, $\mathbf{x}=r\mathbf{y}$ with $r=|\mathbf{x}|$ and $\mathbf{y}\in \mS^2$.  We always assume that $m\ge 0$. Moreover, in the rest of this paper, let
\be
r_*=r+2m\log\lf|\frac{r}{2m}-1\ri|
\ee
for $r>0$, $r\neq 2m$. By convention, $r_*=r$ if $m=0$. We will also denote  $h(r):=1-\frac{2m}r$. We will use the Einstein summation convention.

\subsection{Kruskal extension of Schwarzschild spacetime}\label{ss-extension}

Recall that the original Schwarzschild spacetime is given by  $\mathbb{X}^{3,1}=\R \times( \R^3\setminus \ol B(2m)  )$ where $\ol B(2m)$ is the close Euclidean ball with radius $2m$ and the Schwarzschild metric in polar coordinates in $\R^3$ is given by:
\be\label{e-Sch-metric}
g_\mathrm{Sch}=-\lf(1-\frac{2m}r\ri)dt^2+\lf(1-\frac{2m}r\ri)^{-1}dr^2+r^2 \sigma.
\ee
We consider the Minkowski spacetime as a special case by setting $m=0$.

For $m>0$, the Kruskal extension of Schwarzschild spacetime is given by \cite{Wald}:

\be\label{e-Kruskal-1}
g=\frac{32m^3}r \exp\lf(-\frac r{2m}\ri)\lf(-dT^2+dX^2\ri)+r^2d\sigma^2
\ee
with $r=r(T,X)$ given by
\be\label{e-TXr}
X^2-T^2=( \frac r{2m}-1)\exp(\frac r{2m}).
\ee
with
\be\label{e-TX}
-\infty<X^2-T^2<\infty.
\ee
There are four regions separated by $X^2-T^2=0$.  We only consider the Region I in the notation of \cite{LL2}, where
$X>0$, $X^2-T^2>0$ so that $r>2m$. This corresponds to the original Schwarzschild spacetime:

\be\label{e-I}
\left\{
  \begin{array}{ll}
    T=(\frac r{2m}-1)^\frac12 \exp(\frac r{4m})\sinh (\frac t{4m})=\exp\lf(\frac{r_*}{4m}\ri) \sinh (\frac t{4m}) \\
   X=(\frac r{2m}-1)^\frac12 \exp(\frac r{4m})\cosh(\frac t{4m})=\exp\lf(\frac{r_*}{4m}\ri)\cosh(\frac t{4m}).
  \end{array}
\right.
\ee

\subsection{Conformal compactification}\label{ss-compactification}
In order to understand constant mean curvature surfaces near null-infinity, we will use the usual compactification of the Kruskal extension.

Define $\xi, \chi$, with
\be\label{e-xi-1}
T+X=\tan(\xi+\chi), T-X=\tan(\xi-\chi).
\ee
By \eqref{e-TX}, we have
 \be\label{e-xi-2}
 \left\{
   \begin{array}{ll}
     -\frac\pi2<\xi+\chi, \xi-\chi<\frac\pi2;\\
     -\frac\pi4<\xi<\frac\pi4.
   \end{array}
 \right.
  \ee
Let $N$ denote the above region in $\R^2$, then $N\times {\mathbb{S}}^2$ is diffeomorphic to the Kruskal extension, which can be smoothly extended to the part of $\ol N\setminus N$, consisting of
\be\label{e-infinity}
\left\{
  \begin{array}{ll}
   \mathcal{I}^+: \xi+\chi&= \frac\pi2,\  $ with $0< \xi<\frac\pi 4;\\
   \mathcal{I}^-:  { -}\xi {+}\chi&= \frac\pi2,\  $with $-\frac\pi 4< \xi<0;\\
   \mathcal{I'}^+: \xi-\chi&=\frac\pi2,\ $ with $0\le \xi<\frac\pi 4;\\
   \mathcal{I'}^-: -\xi-\chi&= \frac\pi2,\  $ with $-\frac\pi 4< \xi<0.
  \end{array}
\right.
\ee
Here the sphere is suppressed.  $\mathcal{I}^+$ is the future null-infinity of Schwarzschild spacetime in the region $r>2m$ and  $\mathcal{I}^-$ is the past null-infinity.

\vskip .2cm

Next consider the metric. We are mainly interested in surfaces near the future null-infinity $\mathcal{I}^+$. In the region $r>2m$, consider the retarded null coordinate

\be\label{e-uv}
   \bar v=t-r_*=-4m\log(X-T)=-4m\log\tan(\chi-\xi).
\ee
and let $s=r^{-1}$.
At $ \mathcal{I}^+$, $-\frac\pi 4<\chi-\xi<\frac\pi 4$. So $\bar v$ can be extended smoothly up to the future null-infinity where $s=0$. $y^1,y^2, s, \bar v$ are  local coordinates of the spacetime near $\mathcal{I}^+$, where as before $y^1, y^2$ are local coordinates of $\mS^2$. Sometimes we write $(y^1,y^2, s, \bar v)$ as $(y^1,y^2,y^3,y^4)$. Near $\mathcal{I}^+$, the spacetime is $\mS^2\times(0,s_0)$ for some $s_0>0$, and the  Schwarzschild metric is:
\be\label{e-metric-null}
\begin{split}
g_{\mathrm{Sch}}=&-(1-2ms)d {  \bar v}^2+2s^{-2}  d {  \bar v}ds+s^{-2} \sigma\\
=&s^{-2}(-s^2(1-2ms)d {  \bar v}^2+2d {  \bar v} ds+\sigma)\\
=:&s^{-2}\bar g.
\end{split}
\ee
Here the unphysical metric $\bar g$ is a Lorentz metric near $\mathcal{I}^+$ and is smooth up to $\mathcal{I}^+$. Moreover, $\bar g$ is a product metric:
 \bee
\ol g=(\sigma_{AB})\oplus\left(
      \begin{array}{cc}
        0 & 1 \\
        1 &-s^2(1-2ms) \\
      \end{array}
    \right).
\eee
Here $(\sigma_{AB})$ is the standard metric for $\mS^2$ in local coordinates. In case $m=0$, this is a metric for the future null-infinity of the Minkowski spacetime for such kind of compactification.

\subsection{Mean curvature equations}

\subsubsection{Mean curvature equation near $r=2m$}\label{ss-equation-2m}

Here we assume that $m>0$. In order to study the behavior of the surface near $r=2m$, it is more convenient to use the new variable $\eta$ with
\bee
\eta=h^\frac12.
\eee
Then $dr=m^{-1} \eta r^2d\eta$, $r=2m(1-h)^{-1}=2m(1-\eta^2)^{-1}$. Hence
\be\label{e-metric}
\begin{split}
g_{\mathrm{Sch}}
=&-\eta^2 dt^2+ m^{-2}r^4d\eta^2+r^2 \sigma_{AB}dy^Ady^B
 \end{split}
\ee
Consider a hypersurface $\Sigma$ given as the level set $\{F=0\}$ of a smooth function $F$. Assume $\Sigma$ is spacelike so that $\nabla F$ is timelike. Suppose that $\nabla F$ is future directed. Let $L:=-\la\nabla F,\nabla F\ra$ and consider the future directed unit normal $\mathbf{n}=L^{-\frac12}\nabla F$. Then the mean curvature $H$ of $\Sigma$, which is one-third of the trace of the second fundamental form, is given by
\be\label{e-meancurv-1}
H=\frac13 \mathrm{div}(\mathbf{n})={ \frac13 } \nabla_a(L^{-\frac12}g^{ab}F_b)={ \frac13 } (L^{-\frac12}g^{ab}F_{;ab}-\frac12L^{-\frac32}
\la \nabla L,\nabla F\ra ),
\ee
where $F_{;ab}$ is the Hessian of $F$ with respect to $g_{\mathrm{Sch}}$.

In the following, temporarily denote $x^1=y^1, x^2=y^2, x^3=\eta, x^4=t$. Also, in the following, $a,b,\ldots$ range from 1 to 4 and $i,j,\ldots$ range from 1 to 3, and $A, B, \dots$ range from 1 to 2. Consider a spacelike surface $\Sigma$ which is a graph of a function $u$.
Let $u$ be a smooth function on $\mathbb{R}^{3}$. Denote $F=-t+u$. It is easy to see that $\nabla F$ is in the direction of $\p_t$ and is future directed.

\begin{lma}\label{l-equation-eta} The mean curvature $H$ of $\Sigma$ is given by
\bee
\begin{split}
3HL^\frac32
=&  m^2r^{-4} \lf(L u_{\eta\eta}+(L\eta^{-1}-\frac12L_\eta)  u_\eta \ri)  +r^{-2} \lf(L\wt\Delta u-\frac12 \la\wn L,\wn u\ra  \ri)
\end{split}
\eee
where
$$
L= \eta^{-2}-m^2r^{-4}u_\eta^2-r^{-2}|\wn u|^2.
$$
\end{lma}
\begin{proof} As above, let $F=-t+u$, then $\nabla F$ is future directed.
\bee
 \left\{
  \begin{array}{ll}
  F_i=  u_i; \\
F_4=-1.
  \end{array}
\right.
\eee
Here $F_a, u_i$ are partial derivatives with respect to $a, i$, etc. Then
\bee
 \left\{
  \begin{array}{ll}
   F_{ij}=u_{ij}; \\
F_{4a}=0.
  \end{array}
\right.
\eee
The Christoffel symbols are
\bee
\left\{
  \begin{array}{ll}
   \Gamma_{44}^4=\Gamma_{4A}^4=\Gamma_{33}^4=\Gamma_{3A}^4=\Gamma_{AB}^4=0;\ \  \Gamma_{43}^4= \eta^{-1}\\
\Gamma^3_{AB}=-mr^{-1}{ \eta}\sigma_{AB}; \Gamma_{33}^3=2m^{-1}r  \eta; \Gamma_{44}^3=m^2r^{-4}\eta;\\
\Gamma^3_{A3}=\Gamma^3_{A4}=\Gamma_{34}^3=0;\\
\Gamma^B_{CD}=\wt \Gamma_{CD}^B;\  \Gamma^B_{{ B}3}= { r } m^{-1}\eta;\\
\Gamma^B_{C4}=\Gamma^B_{33}=\Gamma^B_{44}=\Gamma_{34}^B=0;\\
  \Gamma^B_{C3}=0, \ \ \hbox{if $B\neq C$}.
  \end{array}
\right.
\eee
Hence
\be\label{e-meancurv-2}
\begin{split}
g^{ab}F_{;ab}
=& \eta^{-2}\lf({ - }\Gamma_{44}^4+\Gamma_{44}^i u_i\ri)+g^{ij}\lf(u_{ij}{  + }\Gamma_{ij}^4 -\Gamma_{ij}^ku_k\ri)\\
=&\eta^{-2}m^2r^{-4}\eta u_\eta+m^2r^{-4}\lf(u_{\eta\eta}-2{ m^{-1} }r\eta u_\eta\ri)
\\
&+r^{-2}\sigma^{AB}\lf(u_{AB}-\wt\Gamma_{AB}^Cu_C+mr^{-1}{  \eta}\sigma_{AB}u_\eta\ri)\\
=&m^2r^{-4} u_{\eta\eta}+ {   m^2r^{-4}\eta^{-1}u_\eta}+r^{-2}\wt\Delta u
\end{split}
\ee

On the other hand
\bee
\begin{split}
-L=&g^{ab}F_aF_b\\
=&-\eta^{-2}+m^2r^{-4}u_\eta^2+r^{-2}\sigma^{AB}u_Au_B\\
=&-\eta^{-2}+m^2r^{-4}u_\eta^2+r^{-2}|\wn u|^2_\sigma
\end{split}
\eee
Hence
\bee
\begin{split}
\la \nabla L,\nabla F\ra=&-\eta^{-2}L_tF_t+m^2r^{-4}L_\eta F_\eta+r^{-2}\sigma^{AB}L_AF_B\\
=&m^2r^{-4}L_\eta u_\eta +r^{-2}\la\wn L,\wn u\ra_\sigma.
\end{split}
\eee
Thus the result follows.
\end{proof}

\subsubsection{Mean curvature equation near future null-infinity}\label{ss-equation-future-null}
Near the future null-infinity, we use the coordinates which are well-defined and smooth on the compactified spacetime up to $\mathcal{I}^+$, namely $y^1, y^2, s, \bar v$ as in \S\ref{ss-compactification}.  Let $\ol g$ be the unphysical metric. So
\bee
(\ol g)^{-1}=(\sigma^{AB})\oplus
 \left(
      \begin{array}{cc}
      s^2(1-2ms)  & 1 \\
        1 &   0\\
      \end{array}
    \right).
\eee
Here $(\sigma^{AB})$ is the inverse of $(\sigma_{AB})$.

Suppose the surface is given by the graph of $u$, that is $t=u$. In the coordinates $(y^1, y^2, s, \bar v)$, the surface is given by

\begin{equation}\label{defofp}
{ \bar{v}} =u-r_*=-P,	
\end{equation}

in our previous notation. Denote $F={\bar{v}}+P$. Let us compute the mean curvature $H$. Now
$$
d F=P_Ady^A+P_s\p_s+\p_{\bar v}.
$$
so
\be\label{Ldef}
-L:=\la \nabla F,\nabla F\ra_{\ol g}=2 P_s+s^2(1-2ms)P_s^2+|\wt\nabla P|^2.
\ee
Note that $-\nabla F$ is future directed.
\begin{lma}\label{l-equation} The mean curvature equation with respect to the future directed unit normal in this setting is:
\be\label{e-CMC-null-1}
\begin{split}
-3  H L^\frac32=&sL \lf( s^2(1-2ms)P_{ss}+\wt \Delta P\ri)-\frac12 s\lf(L_s+s^2(1-2ms)L_sP_s+\la\wn L,\wn P\ra\ri)\\
&-s^2LP_s-3L
\end{split}
\ee
where
\bee
L=- (2P_s+s^2(1-2ms)P_s^2+|\wt\nabla P|^2 ).
\eee
\end{lma}
\begin{proof}
Suppose
\bee
g=e^{2\lambda}\ol g,
\eee
then
\be\label{e-mean-curv}
 -3H= g^{ab}A_{ab} = e^{-\lambda}\lf(\ol g^{ab}\ol A_{ab}+ 3\la d\lambda,\bar{ \mathbf{n}}\ra_{\ol g}\ri)
= 3 e^{-\lambda}(\ol H+\la d\lambda, \bar{ \mathbf{n}}\ra_{\ol g}),
\ee
where $A, \ol A$ are the second fundamental forms and $H, \ol H$ are the mean curvatures with respect to $g_{\mathrm{Sch}}$ and $\ol g$ respectively.

By \eqref{e-metric-null}, $\lambda=-\log s$. Hence $d\lambda=-s^{-1} ds$. The normal vector is
$\bar{ \mathbf{n}}=L^{-\frac12}\nabla F$. So
\be
\la d\lambda, \bar{ \mathbf{n}}\ra_{\ol g}=-s^{-1}L^{-\frac12}\la ds,dF\ra=-s^{-1}L^{-\frac12}(1+s^2(1-2ms)P_s).
\ee
On the other hand,
\bee
-3\ol H= L^{-\frac12}\lf(\ol g^{ab}+L^{-1}F^aF^b\ri)F_{;ab}\\
= L^{-\frac12} (\Delta_{\ol g}F-\frac12  L^{-1}\la dL,dF\ra_{\ol g} ).
\eee
 Since $\ol g$ is a product metric so that $\ol g=\sigma\oplus  g_1$, where $g_1=-s^2(1-2ms)d {  \bar v}^2+2d {  \bar v} ds$. So

\bee
\begin{split}
\Delta_{\ol g}F=& \wt\Delta P+\Delta_{g_1}F\\
=&\wt \Delta P+2s(1-3ms)P_s+s^2(1-2ms)P_{ss}
\end{split}
\eee

Hence we have
\bee
\begin{split}
 -3H=&s L^{-\frac12}\big(2s(1-3ms)P_s+s^2(1-2ms)P_{ss}+\wt\Delta P-\frac12  L^{-1}\la dL,dF\ra_{\ol g} \big)\\
&-3  L^{-\frac12}(1+s^2(1-2ms)P_s)\\
=&s L^{-\frac12} ( s^2(1-2ms)P_{ss}+\wt\Delta P )-\frac12 s L^{-\frac{3}{2}}\lf(L_s+s^2(1-2ms)L_sP_s+\la \wn L,\wn P\ra\ri)\\
&-s^{2}L^{-\frac12}P_s -3  L^{-\frac12}.
\end{split}
\eee
Multiply the above by $L^\frac32$, the result follows.

\end{proof}

\section{Behaviors of spacelike CMC surfaces at future null-infinity}\label{s-future-null}

In this section, we want to discuss the behavior of a spacelike CMC surface which is the graph of a function $u$, that is $t=u(  \mathbf{y},r)$ with $\mathbf{y}\in \mathbb{S}^2$, $r>r_0$  near the future null-infinity in the compactification as in \S\ref{ss-compactification}. Denote $P(u):= r_*-u$ and consider $P(u)$ as a function of $\mathbf{y}, s$. We always assume the mean curvature of the surface with respect to the future directed unit normal is positive.

\subsection{Spacelike SSCMC surfaces}\label{ss-SSCMC}
 To motivate our study, let us first consider spacelike SSCMC surfaces. By \cite{LL} a spacelike SSCMC with mean curvature $H>0$ in the region I is given by the graph of $f(r)$ defined on $r>2m$ satisfying
\be\label{e-SSeq-1}
f''=3H\lf(h^{-1}-h(f')^2\ri)^\frac 32-\lf[(h^{-1}-h(f')^2)(\frac {2h}r+\frac{h'}2)+\frac{h'}{h}\ri]f'
\ee
and
\be\label{e-SSeq-2}
f'(r)=\frac{\ell}{h \sqrt{1+\ell^2} }=\frac1{h\sqrt{1+\ell^{-2}}}
\ee
where $\ell=\frac{1}{\sqrt h}({  H } r+\frac c{r^2})$. Here $'$ means $\p_r$ for some fixed $r_0>2m$ and $c$ is a constant.  As before   $h=1-\frac{2m}r$.

\begin{prop}\label{p-SS-P} Let $P=P(f)=r_*-f$.
\begin{enumerate}
  \item [(i)] $P$ is smooth up to $s=0$ as a function of $\mathbf{y}, s$.
  \item[(ii)] At $s=0$, $\frac{d^k}{ds^k}P$ can be expressed in terms of $H, m, c$. More precisely, at $s=0,$ $P_s =- \frac{1}{2}{  H^{-2}}$, $P_{ss} =0$, $P_{sss} = \frac{3}{4}{  H^{-4}} $, and $P_{ssss}=H^{-4}(6cH-\frac 92m)$. Moreover,
  $\frac{d^k}{ds^k}P$ for $k\ge 5$ can be expressed in terms of $H, m$ and $P_{ssss}(0)$.

\end{enumerate}

\end{prop}
Before we prove the proposition, we have the following remarks.
\begin{rem}
\begin{enumerate}
  \item [(i)] By adding a constant to a solution $f$ of \eqref{e-SSeq-1}, we may assume that $P(f)=0$ at $s=0$. By the proposition, we see that there are infinitely many spacelike SSCMC   surfaces with the same initial data at infinity up to order 3.
  \item [(ii)] A spacelike SSCMC surface with $P(f)=0$ as $s=0$ will be uniquely determined by the value of $\frac{d^4}{ds^4}P(f)$ at $s=0$.
\end{enumerate}
\end{rem}
\begin{proof}[Proof of Proposition \ref{p-SS-P}]  (i) By \eqref{e-SSeq-2},
\be\label{e-SSCMC-expansion}
\begin{split}
P_s=& s^{-2}\lf(f'-\frac{r}{r-2m}\ri)\\
=&-\frac1{({  H }+cs^3)^2}\cdot \frac{1}{(1+\ell^{-2})^{ \frac12}\lf((1+\ell^{-2})^{ \frac12}+1\ri)}
\end{split}
\ee
with
$$
\ell^{-2}=\frac{s^2(1-2ms)}{(H +cs^3)^2}.
$$
 Hence $P_s$ is integrable near $s=0$ and so $\lim_{s\to0}P$ exists. Moreover, $P_s$ is analytic near $s{=0}$. Hence  (i) is true.

 (ii) The first statement follows immediately from \eqref{e-SSCMC-expansion}.

Now,
\begin{align*}
 \ell^{-2}&=  s^{2}(1-2ms)(  H+ cs^3  )^{-2}\\
 =&H^{-2}\lf(s^{2}-2ms^{3}-2cH^{-1}s^5\ri)+O(s^6).
\end{align*}

\bee
\begin{split}
P_s=&s^{-2}\cdot\frac1{1-2ms}\lf( (1+\ell^{-2})^{-\frac12}-1\ri)\\
=&\sum_{k=0}^3(2ms)^k\cdot\lf(-\frac12s^{-2}\ell^{-2}+\frac38s^{-2}\ell^{-4} \ri)+O(s^4)\\
=&\sum_{k=0}^3(2ms)^k\cdot\lf(-\frac12 H^{-2}+mH^{-2}s+\frac38 H^{-4}s^2+(cH^{-3} -\frac32m H^{-4}) s^3\ri)+O(s^4).
\end{split}
 \eee
From this, it is easy to see that at $s=0$, $P_s=-\frac12 H^{-2}$, $P_{ss}=0$, $P_{sss}=\frac34H^{-4}$ and $P_{ssss}=H^{-4}(6cH-\frac 92m)$.
The last assertion is true because $c$ can be expressed in terms of $H, m$ and $P_{ssss}(0)$.
\end{proof}

We will prove later that the proposition is also true in certain sense for general spacelike CMC surface so that $P$ is smooth up to $s=0$.

We want to investigate more on the structure of SSCMC surfaces.
In case $m=0$, namely, in Minkowski spacetime a complete spacelike SSCMC with mean curvature $H>0$ given by a graph of a solution of  \eqref{e-SSeq-1} is a hyperboloid. It is a hyperbolic space of constant mean curvature $-H$. One expects that in general, a spacelike SSCMC surface in Schwarzschild spacetime is asymptotically to a hyperbolic space of constant mean curvature $-H$. To be precise, let us only consider the case that $H=1$. Following Wang \cite{Wang2001}, a Riemannian manifold $(X^3, g)$ is {\it conformally compact} if $X$ is the interior of a smooth manifold $\ol X$ with boundary $\p X$ such that $\ol g=\tau^2g$ can be  extended to be a $C^3$ metric on $\ol X$ with $\tau>0$ in $X$, $\tau=0$ on $\p X$ and $|d\tau|_{\ol g}>0$.  $(X^3,g)$ is said to be {\it strongly asymptotically hyperbolic} if it is conformally compact with defining function $\tau$ so that near $\tau=0$,
$$
g=\frac1{\sinh^2\tau}\lf(d\tau^2+g_\tau\ri)
$$
where $g_\tau$ is the induced metric on the level set of $\tau$ and
$$
g_\tau=\sigma+\frac{\tau^3}3\zeta+O(\tau^4)
$$
where   $\zeta$ is a symmetric two tensor on $\mS^2$ and the expansion can be differentiated twice. We want to prove:
\begin{prop}\label{p-AH-SSCMC}
Every spacelike SSCMC surface with constant mean curvature 1 is strongly asymptotically hyperbolic in the sense of \cite{Wang2001}. Moreover, the scalar curvature $\cS$ of $\Sigma$ satisfies $|\cS+6|=O(s^6)$ and $\cS+6\ge 0$.
\end{prop}

Before we prove the proposition, let us write down the induced metric on a spacelike surface $\Sigma$ near null-infinity. We use the coordinates $\mathbf{y}, s, \bar v$ for Schwarzschild spacetime as before. Suppose a spacelike surface is given as the zero set of $F=P+\bar v$, where $P=P(\mathbf{y}, s).$ Let $\ol g$ be the unphysical metric. Then
\begin{lma}\label{l-induced-metric} The induced metric of $\ol g$ on $\Sigma$ is
\bee
\begin{split}
\ol g|_\Sigma=&(\sigma_{AB}-s^2(1-2ms)P_AP_B)dy^A  dy^B-(2P_s+s^2(1-2ms)P_s^2)ds^2\\
&
-2(P_A+s^2(1-2ms)P_AP_s)dy^Ads
\end{split}
\eee
and
$$
g_{\mathrm{Sch}}|_{\Sigma}=s^{-2}\ol g|_\Sigma.
$$
Here $y^A$ are local coordinates of $\mS^2$ and $P_A=\frac{\p}{\p y^A}P$, etc.
\end{lma}
\begin{proof} Denote the tangent vectors in Schwarzschild spacetime by $\p_a$, $1\le a\le 4$, where $\p_{A}=\p_{y^A}$ for $A=1, 2$ and $\p_3=\p_s$, $\p_4=\p_{\bar v}$. Then the tangent spaces are spanned by
\be\label{e-basis-Sigma}
\left\{
  \begin{array}{ll}
   e_1=    \p_1-P_1\p_4  \\
    e_2= \p_2-P_2\p_4  \\
   e_3= \p_3-P_3\p_4.
  \end{array}
\right.
\ee
Using the form of $\ol g$, one can see that the lemma is true.
\end{proof}
Note that $\Sigma$ is parametrized by $(y^1,y^2,s)\to (y^1, y^2, s, -P(y^1,y^2,s))$. So $\{e_i\}_{1\leq i\leq 3}$ are just coordinate frames for this parametrization.

Let us prove Proposition \ref{p-AH-SSCMC}.

\begin{proof}[Proof of Proposition \ref{p-AH-SSCMC}]
By Lemma \ref{l-induced-metric}, the induced metric on the SSCMC surface in the setting of the lemma is given by
$$
g_{\mathrm{Sch}}|_\Sigma=s^{-2}\lf(\sigma-(2P_s+s^2(1-2ms)P_s^2)ds^2\ri)
$$
where $P=P(f)$.
By Proposition \ref{p-SS-P}, at $s=0$, $P_s=-\frac12$. Hence $-(2P_s+s^2(1-2ms)P_s^2)\ge C>0$, near $s=0$. Fix $s_0>0$, let
$$
\phi(s)=-\int_s^{s_0}\lambda^{-1}\lf(-(2P_\lambda+\lambda^2(1-2m\lambda)P_\lambda^2\ri)^{\frac12}d\lambda+C
$$
where $C$ is to be determined. By Proposition \ref{p-SS-P}, we have
$$
\lf(-(2P_s+ {s}^2(1-2ms)P_s^2)\ri)^{\frac12}=1-\frac12 s^2-(c-m)s^3+O(s^4)
$$
Hence we may choose $C$ so that
$$
\phi(s)=\log s-\frac14s^2-\frac13(c-m)s^3+O(s^4).
$$
Let
$$
\tau(s)=2 \mathrm{arctanh}(\frac12\exp(\phi(s)))
$$
Then using the fact that $\mathrm{arctanh}(x)=x+ O(x^3)$ near $x=0$, we have
\be\label{e-tau}
\tau(s) =s\lf(1-\frac16 s^2-\frac13(c-m)s^3+O(s^4)\ri).
\ee
Hence $\tau(0)=0$ and $\tau(s)>0$ for small $s>0$.
Moreover
$$
(\sinh\tau)^{-1} \frac{d\tau}{ds}= s^{-1}\lf(-(2P_s+s^2(1-2ms)P_s^2)\ri)^{\frac12}.
$$
Hence

$$
g_{\mathrm{Sch}}|_\Sigma=(\sinh \tau)^{-2}\lf(d\tau^2+w^2(\tau)\sigma\ri).
$$
where
$$
w(\tau)=s^{-1}\sinh \tau.
$$
By \eqref{e-tau},
\bee
\begin{split}
w(\tau)=&\frac\tau s(1+\frac 16\tau^2)+O(\tau^4)\\
=&1-\frac13(c-m)s^3+O(\tau^4).
\end{split}
\eee
\bee
w^2(\tau)=1- \frac23(c-m)s^3+O(\tau^4).
\eee
So $w^2(0)=1$, $\frac{d}{d\tau}w^2|_{\tau=0}=0$, $\frac{d^2}{d \tau^2}w^2|_{\tau=0}=0$, and
\bee
\frac{d^3}{d \tau^3}w^2|_{\tau=0}=-4(c-m).
\eee
Hence
$$
w^2(\tau)= { 1}-\frac23(c-m)\tau^3+O(\tau^4).
$$
From this we can conclude that $\Sigma$ is strongly asymptotically hyperbolic.

To prove the assertion on the scalar curvature $\cS$ of $\Sigma$, it is more convenient to use the original coordinates $t, r, \mathbf{y}$ with $\mathbf{y} { \in \mathbb{S}^{2}}$. Then
\bee
g_{\mathrm{Sch}}=(1+r^2-ar^{-1}+br^{-4})^{-1}dr^2+r^2\sigma=:\rho^2dr^2+r^2\sigma
\eee
with $a=2(m-c)$, $b=c^2$.  The mean curvature $H$ of the level set $r$=constant with respect to the normal $\mathbf{n}=\rho^{-1}\p_r$ is given by
$$
H=2(r\rho)^{-1}.
$$
By the Gauss equation and the formula for $\p_{\mathbf{n}}H$, we have
$$
\cS=2r^{-2}-2\rho^{-1}\p_rH-H^2-|A|^2=2r^{-2}-2r\p_r((r\rho)^{-2})-6(r\rho)^{-2}.
$$
where $A$ is the second fundamental form of the level set, because the level set is umbilical. On the other hand,
$$
(r\rho)^{-2}=1+r^{-2}-ar^{-3}+br^{-6}.
$$
From this, it is easy to see that
\bee
\begin{split}
\cS
 =&2r^{-2}+4r^{-2}-6ar^{-3}+12b r^{-6}-6(1+r^{-2}-ar^{-3}+br^{-6})\\
 =&-6+6br^{-6}.
\end{split}
\eee
Thus the result follows.

\end{proof}

\begin{rem}\label{r-mass}
Direct computation shows that  the mass integral of the SSCMC surface is given by:
$(4(m-c),0,0).$  In particular, it is future directed if and only if $m>c$.

\end{rem}

\subsection{General spacelike CMC surfaces (I)}\label{ss-CMC-future-1}
We want to obtain similar results for general spacelike CMC surfaces. Let $\Sigma$ be a spacelike surface which is a graph of a function $P$ over $(\mathbf{y},s)$  as in the beginning of this section.  Namely, $\Sigma$  is the zero set of $F=P+\bar v$ and $P=P(\mathbf{y},s)$. Motivated by the results in the SSCMC case, in this subsection, we discuss the conditions for $P$ being smooth at some open set containing a point $(  \mathbf{y},0 )\in \mathcal{I}^+ $ near the null-infinity. We want to find the relation between $\p_s^kP$ at $s=0$ and the data $P( \mathbf{y},0)$ at $s=0$.
Let $H>0$ be the mean curvature of $\Sigma$ with respect to the future directed unit normal.
We do not assume $H$ to be constant, because in this setting the results could be applied to spacelike CMC surfaces in spacetimes which are close to Schwarzschild spacetime near the null-infinity. Such a surface will not have constant mean curvature in the Schwarzschild spacetime. But it will have constant mean curvature at $s=0$ up to certain order.

 In the following, for any $k>0$, use $\phi=\mathcal{O}(s^k)$ to mean that it is a function on $\Sigma$ so that $s^{-k}\phi$ can be smoothly extended to $s=0$. We also denote a function with this property by $\mathcal{O}(s^k)$. As before, $L=-\la \nabla F,\nabla F\ra$.

\begin{thm}\label{t-smooth-future} Assume $P$ and $H$ are $C^4$ up to $s=0$, and let $H_0=H|_{s=0}$. Then the following are true.
\begin{enumerate}
  \item[(a)] At $s=0$, $L=0$ or $H_0^{-2}$.

  \item [(b)] Assume $L= H_0^{-2}>0$ at $s=0$. Then

\begin{enumerate}
  \item [(i)] At $s=0$, $P_s=-\frac12(  H_{0}^{-2} +|\wn P|^2)$.
  \item[(ii)] $H=H_0+\mathcal{O}(s^2)$  if and only if $L_s=-H_0^{-2}\lf(\wt\Delta P+H_0^{-1}\la \wn H_0,\wn P\ra\ri)$ at $s=0$,  which is equivalent to
 \bee
P_{ss}=\frac12 H_0^{-2}\wt \Delta P+\frac12\la\wn|\wn P|^2,\wn P\ra+ \frac12H_0^{-3}\la \wn H_0, \wn P\ra.
\eee
\item[(iii)] Suppose $H=H_0+ \mathcal{O}(s^2)$, then $H=H_0+\mathcal{O}(s^3)$  if and only if
$$
L_{ss}
=-\frac92 H_0^{-2} L_s^2-4(L_s\wt\Delta P+ L\wt\Delta P_s)  +2 \la \wn L_s,\wn P\ra+2\la \wn (H_0^{-2}),\wn P_s\ra+4H_0^{-2}P_s
$$
at $s=0$,
which is equivalent to
\bee
\begin{split}
P_{sss}=&\frac94 H_0^{-2} L_s^2+2(L_s\wt\Delta P+ \wt\Delta P_s)  -  \la \wn L_s,\wn P\ra-\la \wn (H_0^{-2}),\wn P_s\ra-2H_0^{-2}P_s\\
&-\la \wn P_{ss}, \wn P\ra- \la \wn P_s,\wn P_s\ra -  P_s^2 .
\end{split}
\eee
where $L_s, P_s, P_{ss}$ are given by (b)(i)(ii) which can be expressed in terms of $H$ and the boundary value of $P$ at $s=0$ and their  derivatives as  functions in $\mS^2$.

\end{enumerate}

\item[(c)] Assume $L=   H_0^{-2}$ and $H=H_0+\mathcal{O}(s^3)$. Then the boundary value of $P$  at $s=0$ satisfies
  the following compatibility condition at $s=0$,
\bee
\begin{split}
0=& -2H_0^{-3}H_{sss}+ \frac{9}2  H_0^{-1}\la \wn H_0,\wn P\ra  L_{ss}+ \frac{5}2  L_{ss}\wt\Delta P+ \frac34  H_0^{4} L_s^3
   -4L_s\wt\Delta P_s-2H_0^{-2}\wt\Delta P_{ss}\\
&+6L_sP_s+ \la\wn L_{ss},\wn P\ra+2\la\wn L_s,\wn P_s\ra+\la \wn (H_0^{-2}),\wn P_{ss}\ra
\end{split}
\eee
 where $P_s, L_s, P_{ss}$, and $L_{ss}$ are given by (b)(i)--(iii), which can be expressed in terms of $H$ and the boundary value of $P$ at $s=0$ and their derivatives as functions on $\mS^2$.

\item[(d)] Assume $P$ and $H>0$ is smooth up to $s=0$.  Suppose $L=  H_0^{-2}  $ at $s=0$,   then at $s=0$, $\p_s^kP$ can be expressed in terms of $P$, $P_{ssss}$ and their derivatives with respect to $\mathbf{y}$ at $s=0$, and $H$, $\p^{k-1}_sH$ at $s=0$.
\end{enumerate}

\end{thm}
\begin{proof}

Let $s=0$ in  \eqref{e-CMC-null-1}, then we have
$$
9  H_0^{2}  L^3=9L^2
$$
and (a) follows.

To prove (b), differentiating $L$ with respect to $s$, we have at $s=0$
\be\label{e-Ls}
\left\{
  \begin{array}{ll}
    L_s= -2P_{ss}-2\la\wn P_s,\wn P\ra;\\
L_{ss}= - \lf (  2P_{sss}+2\la \wn P_{ss}, \wn P\ra+2\la \wn P_s,\wn P_s\ra +2 P_s^2\ri );\\
L_{sss}=- \lf(2P_{ssss}+2\la \wn P_{sss},\wn P\ra+6\la\wn P_{ss},\wn P_s\ra+12P_sP_{ss}-12mP_s^2 \ri).
  \end{array}
\right.
\ee
This give relation between $L_s$ and $P_{ss}$, etc.

Assume $L=H_0^{-2}>0$ at $s=0$. Then at $s=0$, we have
$$
P_s=-\frac12(L+|\wn P|^2)=-\frac12(H^{-2}_0+|\wn P|^2).
$$
This proves (b)(i).
Differentiate \eqref{e-CMC-null-1} once and let $s=0$.
We have
\bee
-3H_sL^\frac32-\frac92H_0L^\frac12 L_s=L\wt\Delta P-\frac12\lf(L_s+\la\wn L,\wn P\ra\ri)-3L_s
\eee

So
\bee
L_s=-H_0^{-2}\wt\Delta P-H_0^{-3}\la \wn H_0,\wn P\ra-3H_sL^\frac32.
\eee
Hence $H=H_0+\mathcal{O}(s^2)$ if and only if
$$
L_s=-H_0^{-2}\wt\Delta P-H_0^{-3}\la \wn H,\wn P\ra.
$$
Combining with \eqref{e-Ls}, we conclude that (b)(ii) is true.

By $H=H_0+\mathcal{O}(s^2)$, then $H_s=0$.  Differentiating \eqref{e-CMC-null-1} twice with respect to $s$ and let $s=0$, and using $L=H_0^{-2}$, we have
\bee
\begin{split}
&-3H_{ss}L^\frac32-\frac 94\lf(  H_0^{2}L_s^2+2L_{ss}\ri)\\=&2(L_s\wt\Delta P+ L\wt\Delta P_s)-L_{ss}-   \la \wn L_s,\wn P\ra-\la\wn L,\wn P_s\ra -2H_0^{-2}P_s-3L_{ss}.
\end{split}
\eee
So
\bee\label{e-2nd-order}
-\frac12 L_{ss}=\frac94 H_0^{-2} L_s^2+2(L_s\wt\Delta P+ \wt\Delta P_s)  -  \la \wn L_s,\wn P\ra-\la \wn L,\wn P_s\ra-2H_0^{-2}P_s+3H_{ss}L^\frac32.
\eee
From this and by \eqref{e-Ls}, we conclude that (b)(iii) is true.

To prove (c), assume $H=H_0+\mathcal{O}(s^3)$ so that $H_s=H_{ss}=0$ at $s=0$. We want to prove that at any point $\mathbf{y}\in \mS^2$, the equality in the statement must be  true at $(\mathbf{y},0)$. So fixed $\mathbf{y}\in \mS^{2}$. Differentiate \eqref{e-CMC-null-1} three times with respect to $s$, we have
\bee
\begin{split}
\p_s^3\mathbf{LHS}\to -3H_{sss}L^\frac32-\frac92 (L_{sss}+\frac32 H_0^{2} L_sL_{ss}-\frac14  H_{0}^{4} L_s^3 )
\end{split}
\eee
as $s\to 0$, where the limiting function is evaluated at $(\mathbf{y},0)$. The limit exists, because each term involves the derivatives of $P$ up to at most fourth order. On the other hand, at $(\mathbf{y},s)$
\bee
\begin{split}
\p_s^3\mathbf{RHS}=\mathrm{I}+\mathrm{II}
\end{split}
\eee
where $\mathrm{II}$ involves only derivatives of $P$ up to at most fourth order and $\mathrm{I}$ may involve derivatives up to fifth order. More precisely,
\bee
\begin{split}
\mathrm{I}=&sL\lf(s^2(1-2ms)P_{sssss}+\wt\Delta P_{sss}\ri)\\
& -\frac12 s\lf((1+s^2(1-2ms)P_s)L_{ssss} +\la\wn L_{sss},\wn P\ra\ri) .
\end{split}
\eee
and
\bee
\mathrm{II}\to 6LP_{ss}+3(L\wt\Delta P)_{ss}-\frac32 L_{sss}-3L_sP_s-\frac32(\la\wn L,\wn P\ra)_{ss}-6(LP_s)_s-3L_{sss}
\eee
as $s\to 0$, where the limiting function is evaluated at $(\mathbf{y},0)$. Again, the limit exists, because each term involves the derivatives of $P$ up to at most fourth order.
\vskip .1cm

{\it Claim}: We can find $s_i\to 0$ such that
\bee
\mathrm{I}(\mathbf{y},s_i)\to 0.
\eee

\vskip .1cm

Suppose the claim is true. Then we can choose such $s_i$ and let $s_i\to 0$ to conclude that
at $(\mathbf{y},0)$:
\bee
\begin{split}
&-3H_{sss}L^\frac32-\frac92 (L_{sss}+\frac32 H_0^{2} L_sL_{ss}-\frac14  H_{0}^{4} L_s^3 )\\
=&6LP_{ss}+3(L\wt\Delta P)_{ss}-\frac32 L_{sss}-3L_sP_s-\frac32(\la\wn L,\wn P\ra)_{ss}-6(LP_s)_s-3L_{sss}\\
\end{split}
\eee
Note that the term $L_{sss}$ will be cancelled out. Hence we have
\bee
\begin{split}
0=& 3H_0^{-3}H_{sss} +\frac{27}4  H_0^{2} L_sL_{ss}-\frac98  H_0^{4} L_s^3  +3L_{ss}\wt\Delta P\\
&+6L_s\wt\Delta P_s+3H_0^{-2}\wt\Delta P_{ss}-9L_sP_s\\
&-\frac32\lf(\la\wn L_{ss},\wn P\ra+2\la\wn L_s,\wn P_s\ra+\la \wn (H_0^{-2}),\wn P_{ss}\ra\ri)
\end{split}
\eee
Combing with (b)(ii), the result (c) follows.

To prove the claim, for fixed $\mathbf{y}\in \mS^{2}$, let
\bee
\begin{split}
f(s)=& L\lf(s^2(1-2ms)P_{ssss}+\wt\Delta P_{ss}\ri)\\
& -\frac12  \lf((1+s^2(1-2ms)P_s)L_{sss} +\la\wn L_{ss},\wn P\ra\ri) .
\end{split}
\eee
Then $f(s) $ is continuous at $s=0$, because $P$ is $C^4$. Hence
$$
f(s)-f(\frac s2)=o(1),
$$
as $s\to 0$. Therefore, there exists $\xi\in (\frac s2,s)$, so that
$$
f_s(\xi)\cdot \frac s2=o(1),
$$
as $s\to 0$. This implies that $f_s(\xi)\cdot \xi=o(1)$ as $s\to 0$. Hence there exists a sequence $s_i\to 0$, such that
$$
f_s(s_i)\cdot s_i\to 0.
$$
But
\bee
\begin{split}
s_if_s(s_i)
=&\mathrm{I}(\mathbf{y},s_i)+s_iQ(\mathbf{y},s_i) .
\end{split}
\eee
where $Q$ involves derivatives of $P$ only up to fourth order. Since $P$ is $C^4$, it is easy to see that the claim is true.

To prove (d), let us take the square of both sides of   \eqref{e-CMC-null-1} and rewrite it as follows:
\be\label{e-CMC-null-2}
\begin{split}
9H^2L^3
=&9L^2+3sLL_s-6s L \sum_{k=1}^4a_ks^{k-1}+s^2 (\frac12L_s- \sum_{k=1}^4a_ks^{k-1} )^2
\end{split}
\ee
where $L=-\lf(2P_s+s^2(1-2ms)P_s^2+|\wn P|^2\ri)$ and
\bee
\left\{
  \begin{array}{ll}
    a_1= L\wt \Delta P-\frac12\la\wn L,\wn P\ra;\\
a_2=-LP_s;\\
a_3=LP_{ss}-\frac12L_sP_s;\\
a_4=-2mLP_{ss}+mL_sP_s
  \end{array}
\right.
\eee
First note that
\be\label{e-L-11}\p_s^jL=-2\p_s^{j+1}P+Q^{(j)},
 \ee
here and below $Q^{(j)}$ will denote a term involving only $\p_s^{i}P$ and their derivatives with respect to $\mathbf{y}$ for $0\le i\le j$. It may vary from line to line.

Suppose $L={  H_0^{-2} }$ at $s=0$.
For $j\ge 1$, differentiate \eqref{e-CMC-null-2} $j$-times   with respect to $s$, and let $s=0$, then
\bee
\p_s^j\lf(\mathbf{LHS}\ri)=27 H_0^{-2} \p_s^j L +Q^{(j)}+R^{(j)}.
\eee
here and below, $R^{(j)}$ will denote a term involving the $\p_s^{i}H$ at $s=0$, with $0\le i\le j$, and $\p_s^{i}P$ and their derivatives with respect to $\mathbf{y}$ for $0\le i\le j$.
\bee
\p_s^j\lf(\mathbf{RHS}\ri)=18{  H_0^{-2}}\p_s^jL+3j{  H^{-2} }\p_s^jL+Q^{(j)}.
\eee
Hence
\bee
(j-3) \p_s^jL=Q^{(j)}+R^{(j)}.
\eee
From the proof of part (b), it is easy to see that (d) is true for $j=0, 1, 2, 3$ with $\p_s^jP$ depends only on $P$, and $\p_s^iH$, for $0\le i\le j$.
If $\p_s^4P$ is known for $s=0$, then one can see from the above that (d) is true by induction.

\end{proof}

The condition that $L={   H_0^{-2} }$ in Theorem \ref{t-smooth-future} will be satisfied in the most interesting case. Namely, let $\Sigma$ be a spacelike surface with positive mean curvature $H$ defined near the future null-infinity as a graph of $u=u(\mathbf{y}, r)$ with $\mathbf{y}\in {{\mathbb{S}^2}}$. Let $P(u)=r_*-u$. Then the surface is given by $P(u)+\bar v=0$ with respect to the coordinates of the compactified spacetime. We have the following.
\begin{prop}\label{p-L=1} With the same notations as in Theorem \ref{t-smooth-future},
suppose $P(u)$ is smooth near the whole $\mathbb{S}^2\times\{0\}$ and $H>0$ is smooth up to $s=0$. Then $L\equiv  H_0^{-2}$ at $s=0$.
\end{prop}
\begin{proof} By Theorem \ref{t-smooth-future}(a), at each point we have either $L=0$ or $L=H_0^{-2}$ at $s=0$. Since $H_0>0$ and is continuous at $s=0$, $H_0\ge C>0$ on $\mathbb{S}^2\times\{0\}$. Hence $L\equiv0$ at $s=0$ or $L \equiv H_0^{-2}$ at $s=0$. To prove the proposition, it is sufficient to prove that $L>0$ at a some point in $\mathbb{S}^2\times\{0\}$. In order to prove this, we use the spacelike SSCMC surfaces constructed in \cite{LL} as barriers. To be precise, suppose we can find a SSCMC surface with constant mean curvature $H_1>0$ given as a graph of some function $f$. By \S\ref{ss-SSCMC}, $P(f)$ is smooth up to $s=0$ and $L(f)=-2P_s(f)-|\wn P(f)|^2=-2P_s(f)= H_1^{-2}.$ We denote $L$ by $L(u)$ to be more precise.  Suppose  $f\le u$ near $s=0$ and $f=u$ at some point $p$ of $\mathbb{S}^2\times\{0\}$ so that $\wn P(u)=0$ at this point. Then $P(f)\ge P(u)$ near $s=0$ and $P(f)(p)=P(u)(p)$. Hence at $p$, we have $P_s(f)\ge P_s(u)$ and so at this point:
$$
H_1^{-2} =L(f)\le -2P_s(u)=L(u)+|\wn P(u)|^2=L(u).
$$
From this the result follows. Hence it remains to find a suitable barrier function and a suitable point $p$.

Let $p\in  \mathbb{S}^2\times\{0\}$ such that
$$
P(u)(p)=\max_{s= 0 }P(u)=:a.
$$
Then $\wn P(u)=0$ at $p$.
Next let us find a suitable spacelike SSCMC surface as barrier. Since $H>0$ and smooth up to $s=0$, then there exists $s_0>0$  and $H_1>0$ such that $H_1>H$ in $0\le s\le s_0$.
Consider the spacelike SSCMC surface which is the graph of $f=f(r)$ such that (see \cite{LL})

$$
f'=f_r=\frac{\ell}{h\sqrt{1+\ell^2}}; \ \ \ell=\frac1{\sqrt h}( {  H_1}r+cr^{-2})
$$
 for  $c>-8m^3{  H_1}$  to be chosen. The graph of $f$ has constant mean curvature $H_1$. Here as before $h=1-\frac{2m}r$. Let $r_0=s_0^{-1}$. Then
\bee
\begin{split}
P(f)(0) -P(f)(s_0)=&\int_{r_0}^\infty P'(f) dr\\
=&\int_{r_1}^\infty    \frac1{({  H_1}r+cr^{-2})^2}\frac1{\sqrt{1+\ell^{-2}}(1+\sqrt{1+\ell^{-2}})}dr\\
\le &\int_{r_0}^\infty \frac1{({  H_1}r+cr^{-2})^2}dr\\
\le &\int_{r_0}^{r_1}\frac1{({  H_1}r+cr^{-2})^2}dr+\int_{{ r_1}}^{\infty}\frac1{r^2}dr\\
=&\int_{r_0}^{r_1}\frac1{({  H_1}r+cr^{-2})^2}dr+\frac1{r_1}.
\end{split}
\eee
 for any $r_1>r_0$.
 For any $\e>0$, we can choose $r_1>r_0$ large enough so that $1/r_1<\e$. Then choose $c$ large enough so that the first integral is less than $\e$. Hence for any $\e>0$, we can choose $c>>1$ such that
 $$
 P(f)(0)-P(f)(s_0)\le \e.
 $$
Let $b=\max_{\mathbf{y}\in \mS^2}P(u)(\mathbf{y},s_0)=P(u)(\mathbf{y}_0,s_0).$ Since $\Sigma$ is spacelike, we have $P_s(u)>0$. Hence
$$
b=P(u)(\mathbf{y}_0,s_0)<P(u)(\mathbf{y}_0,0)\le a.
$$
Choose $c>>1$, we can find $f$ such that
$$
P(f)(0)-P(f)(s_0)\le \frac12(a-b)
$$
and $P(f)(0)=a=P(u)(p)$ by adding a constant to $f$ if necessary. By our construction, we have
\bee
\left\{
  \begin{array}{ll}
    P(f)(0)=a\ge P(u)(\mathbf{y},0);\\
P(f)(s_0)>b\ge P(u)(\mathbf{y},s_0)
  \end{array}
\right.
\eee
for all $\mathbf{y}\in \mS^2$, because
$$
P(f)(s_0)\ge P(f)(0)-\frac12(a-b)=\frac12(a+b)>b.
$$
Moreover, the graph of $f$ has mean curvature $H_1>H$ for $0<s<s_0$. By the definition of $P(u), P(f)$, we have $u\ge f$ both at $s=0, s=s_0$. Since the mean curvature equation in the region $r>2m$ is elliptic for spacelike surfaces, we conclude that $u\ge f$ for $0<s<s_0$.
This completes the proof of the proposition.

\end{proof}
  In case $H_0$ is constant, the results are more clean. Combining with Theorem \ref{t-smooth-future} and the notations as in the theorem, we have:

\begin{cor}\label{c-smooth-null} Assume $H=H_0+\mathcal{O}(s^3)$, where $H_0>0$ is a constant. Suppose $\Sigma$ is a spacelike CMC surface given as the graph of $P$ as in Theorem \ref{t-smooth-future} and $P$ is $C^{4}$ up to $s=0$. Then at $s=0$, $L=H_0^{-2}$, and
\be \label{e-smooth-2}
\left\{
  \begin{array}{ll}
  P_s=-\frac12({  H_0^{-2} }+|\wn P|^2),\\
  L_{s}= -H_0^{-2}\wt\Delta P,\\
  P_{ss}=  \frac{1}{2}\la\wn ( |\wt\nabla P|^2 ),\wn P\ra+ \frac{1}{2}H_0^{-2}\wt\Delta P ,\\
  L_{ss}=-\frac12 H_0^{-2}(\wt\Delta P)^2+2H_0^{-2}{\wt\Delta}  |\wn P|^2-2H_0^{-2} \la \wn \wt \Delta P, \wn P\ra-2H^{-2}_0(H_0^{-2}+|\wn P|^2)\\
P_{sss}= \frac14 H_0^{-2}(\wt\Delta P)^2- H_0^{-2} \wt\Delta  |\wn P|^2+ H_0^{-2} \la \wn \wt \Delta P, \wn P\ra+ H^{-2}_0(H_0^{-2}+|\wn P|^2)\\
\ \ \  -\la \wn P_{ss}, \wn P\ra- \la \wn P_s,\wn P_s\ra -  P_s^2.
  \end{array}
\right.
\ee
Moreover, if $H=H_0+ \mathcal{O}(s^4)$, then $P$ satisfies the compatibility condition at $s=0$,
\bee
\begin{split}
0=&    \frac{5}2  L_{ss}\wt\Delta P+ \frac34  H_0^{4} L_s^3-4 L_s\wt\Delta P_s-2H_{0}^{-2}\wt\Delta P_{ss}+6L_sP_s
 +  \la\wn L_{ss},\wn P\ra+2\la\wn L_s,\wn P_s\ra
\end{split}
\eee
where $L_s, P_s, L_{ss}$ and $P_{ss}$ are given by \eqref{e-smooth-2}.

Conversely, if $P$ satisfies \eqref{e-smooth-2} and $H=H_0$ at $s=0$, then $H=H_0+\mathcal{O}(s^3)$ near $s=0$.

\end{cor}
\begin{rem} \begin{enumerate}
              \item [(i)] By the corollary, if $H$ is constant for example, then $P_s$, $P_{ss}$, $P_{sss}$ at $s=0$ can be determined by the boundary data of $P$ at $s=0$. On the other hand,  \S\ref{ss-SSCMC} shows that   $P_{ssss}$ cannot be expressed in terms of the boundary data in general.

              \item [(ii)] By part (d) of  Theorem \ref{t-smooth-future}(d), if two spacelike CMC surfaces with the same mean curvature so that they have the same boundary data at $s=0$, then they will be the same up to infinity order at $s=0$ provided the two corresponding $P_{ssss}$ are the same at $s=0$.
            \end{enumerate}

\end{rem}

\begin{rem}\label{r-compatibility} The compatibility condition is rather complicated.  If $P$ is constant,  it is easy to see that it satisfies the compatibility condition in   the above corollary. If $P=\mathbf{a}\cdot \mathbf{x}$ for some constant vector $\mathbf{a}\neq 0$ and $\mathbf{x}\in \mathbb{S}^{ 2}$, then it also satisfies the compatibility condition with $H_0=1$ either by direct computations or by the fact that it is the boundary data of ${ r-u}$  in the Minkowski spacetime where
$$
u(x)=\lf(1+|\mathbf{x}+\mathbf{a}|^2\ri)^\frac12.
$$
whose graph has constant mean curvature 1. It is unclear if there are other functions which satisfy the compatibility condition.
\end{rem}

\subsection{General spacelike CMC surfaces (II)}\label{ss-CMC-future-2}
We continue to use the setting in the previous section to study further behavior of general spacelike CMC surfaces near the future null-infinity.
In Proposition \ref{p-AH-SSCMC} in \S\ref{ss-SSCMC}, we show that a spacelike SSCMC surface is strongly asymptotically hyperbolic and the scalar curvature is asymptotically -6 up to order $s^6$. We want to understand the boundary behavior under similar additional condition of the scalar curvature. It turns out that this will put more restrictions on $P$.  We only consider the case that the mean curvature is asymptotically 1.

Let $\Sigma$ be a spacelike CMC surface near the future null-infinity so that it is the graph of $P=P(\mathbf{y},s)$  as before. In this subsection, we make the following assumptions. Let $\cS$ be the scalar curvature of $\Sigma$.

\begin{ass}\label{a1} $u$ is defined near $s=0$ such that $P$ is smooth up to $s=0$. Moreover, the mean curvature of the graph of $u$ is $H=1+\mathcal{O}(s^3).$

\end{ass}

\begin{thm}\label{t-mathringA} Under Assumption \ref{a1},  $|\mathring A|_G=O(s)$.

\begin{enumerate}
  \item[(a)] Moreover, the following statements are equivalent:
  \begin{enumerate}
    \item[(i)] $|\mathring A|_G =O(s^{2})$.
    \item[(ii)] $P$ is the linear combination of the first four eigenfunctions of the Laplacian of $\mS^2$ with respect to the standard metric.
    \item [(iii)] $|\mathring A|_G =O(s^{3})$.
  \end{enumerate}
    In case $H=1$, then they are equivalent to the fact that $|\cS+6|=O(s^6)$.
  \item [(b)]
    Furthermore, if $m>0$ then the following two statements are equivalent:
  \begin{enumerate}
    \item[(i)] $|\mathring A|_G =O(s^{4})$.
    \item[(ii)] $P$ is constant and $P_{ssss}=-\frac{9m}2$ at $s=0$.
  \end{enumerate}
\end{enumerate}
\end{thm}

The last statement of part (a) of  the theorem  follows immediately from the following observation:

\begin{lma}\label{l-scalar} Suppose $\Sigma$ is a spacelike surface    defined near the null-infinity of the Schwarzchild spacetime with constant mean curvature $H=1$.  Let $\cS$ be the scalar curvature of $\Sigma$ and let $\mathring A$ be the traceless part of the second fundamental form $A$ of $\Sigma$. Then
 $\cS+6=\mathcal{O}(s^k)$ if and only if $|\mathring A|^2=\mathcal{O}(s^k)$ for $k>0$.
Here the norm of $\mathring A$ is computed with respect to the induced metric on $\Sigma$.
\end{lma}
\begin{proof} By the Gauss equation, and the fact that the Ricci tensor of Schwarzschild spacetime is zero, we have

\bee
\cR=-9 +|  A|^2=-6+|\mathring A|^2
\eee
because the mean curvature is 1, which is one-third of the trace of $A$. From this the result follows.
\end{proof}
\begin{rem}\label{r-scalar} In general, if the spacetime has a compactification of the future null-infinity similar to Schwarzschild spacetime and if $s^{-k+2}\Ric$ can be extended to be a smooth tensor  up to $s=0$, then the result of the lemma is still true.
\end{rem}
To prove the theorem, let us first fix some notations. Let $y^1, y^2$ be coordinates of $\mS^2$ and let  $s=y^3, \bar v=y^4$. In the following, $\a, \b,\dots$ range from 1 to 2; $i, j,\dots $ range from 1 to 3; and  $a, b,\dots$ range from 1 to 4. Suppose $\phi$ is a function on the surface, $\phi_a$ will  denote  the partial derivative  with respect to $y^a$.
We will denote the tangent vector  $\p_{y^a}$ of the ambient space  by $\p_a$.  Let $\ol g$ be the unphysical metric as before. Direct computations give:

\begin{lma}\label{l-connection} Choose a normal coordinates $y^1, y^2$ at a point $\mathbf{y}_0 \in \mathbb{S}^2$ and let $\ol \nabla$ be the covariant derivative with respect to $\ol g$. Then
\be\label{e-connection}
\left\{
  \begin{array}{ll}
    \ol\nabla_{\p_\b}\p_a=  \ol \nabla _{\p_a}\p_\b=0,\ 1\le a\le 4, 1\le \b\le 2;\\
\ol\nabla_{\p_3}\p_3=0; \\
\ol\nabla_{\p_4}\p_4
=s^3(1-5ms+6m^2s^2)\p_3+s(1-3ms)\p_4\\
\ol\nabla_{\p_3}\p_4=\ol\nabla_{\p_4}\p_3=-s(1-3ms)\p_3.
  \end{array}
\right.
\ee
\end{lma}

\vskip.1cm

Let $\Sigma$ be parametrized by
$$
\Phi:(y^1,y^2,s)\to (y^1,y^2,s,-P(y^1,y^2,s)).
$$
Choose the following basis in $T(\Sigma)$:
\be\label{e-basis}
\left\{
  \begin{array}{ll}
   e_1=    \p_1-P_1\p_4;  \\
    e_2= \p_2-P_2\p_4; \\
   e_3= \p_3-P_3\p_4.
  \end{array}
\right.
\ee
Note that they are just coordinate frames with respect to the coordinates $y^1, y^2, s$.
\begin{lma}\label{l-induced-connection} At a point on $\Sigma$, choose a normal coordinates $y^1, y^2$ at a point  $\mathbf{y}_0\in\mS^2$.
\begin{enumerate}
  \item [(i)] Let $\ol g_{ij}:=\ol g(e_i,e_j)$, then
\bee
\left\{
  \begin{array}{ll}
    \ol g_{\a\b}=\delta_{\a\b}-s^2(1-2ms)P_\a P_\b,\ \  1\le \a,\b\le 2;\\
\ol g_{33}=-s^2(1-2ms)P_s^2-2P_s;\\
\ol g_{3\a}=-P_\a-s^2(1-2ms)P_\a P_s,\ \  1\le \a\le2.\\
  \end{array}
\right.
\eee
The metric on $\Sigma$ induced by $g_{\mathrm{Sch}}$ is given by
\bee
G_{ij}:=g_{\mathrm{Sch}}(e_i,e_j)=s^{-2}\ol g_{ij}.
\eee
  \item [(ii)] \bee
\left\{
  \begin{array}{ll}
  \ol\nabla_{e_\a}e_\b=-P_{\a\b}\p_4+P_\a P_\b \ol\nabla_{\p_4}\p_4, 1\le \a,\b\le 2; \\
  \ol\nabla_{e_3}e_3= 2s(1-3ms)P_{s}\p_3-P_{ss}\p_4+P_ s ^2\ol\nabla_{\p_4}\p_4;\\
  \ol\nabla_{e_3}e_\b=-P_{s\b}\p_4+ s(1-3ms)P_\b\p_3+P_s P_\b \ol\nabla_{\p_4}\p_4, 1\le \b\le2.
  \end{array}
\right.
\eee

\end{enumerate}

\end{lma}
\begin{proof} By Lemmas \ref{l-induced-metric}, \ref{l-connection} and direct computations, the results follow.

\end{proof}

By Theorem \ref{t-smooth-future} and Corollary \ref{c-smooth-null}, $L=-\la \nabla F,\nabla F\ra$ is 1 at $s=0$. In particular, it is positive near $s=0$.
\begin{lma}\label{l-2nd} In the setting as in the previous lemma, at the point $(\mathbf{y}_0, y^3,y^4)$, let  $\ol A$ and $A$ be the second fundamental forms of $\Sigma$ with respect to the future directed unit normal and with respect to the metrics $\ol g$ and $g_{\mathrm{Sch}}$ respectively, and let $\ol A_{ij}=\ol A(e_i,e_j)$, $A_{ij}=A(e_i,e_j)$, where $e_i$ are as in \eqref{e-basis}. Then
\bee
\left\{
  \begin{array}{ll}
  \ol A_{\a\b} =L^{-\frac12}\lf(-3mP_\a P_\b s^2+P_\a P_\b s-P_{\a\b}\ri)+O(s^3),  1\le \a,\b\le 2;\\
\ol A_{33}=L^{-\frac12}\lf(  - 9m P_s^2s^2  +3P_s^2s-P_{ss}\ri)+O(s^3);\\
\ol A_{\b3}=L^{-\frac12}\lf( -6m P_s P_\b s^2+2 P_\b P_s s-P_{\b s}\ri)+O(s^3), 1\le \b\le 2. \\
 \end{array}
\right.
\eee

Moreover
$$
A_{ij}=s^{-1}\lf(\ol A_{ij}+B\ol g_{ij}\ri)
$$
where $L=-\lf(2P_s+s^2(1-2ms)P_s^{2}+|\wn P|^2\ri)$ and $B=s^{-1}L^{-\frac12}\lf(1+s^2(1-2ms)P_s\ri)$.
\end{lma}
\begin{proof} The unit future pointing normal of $\Sigma$ with respect to $\ol g$ is:

$$
\nu=-L^{-\frac12}\nabla F=-L^{-\frac12}\lf(\sum_\a P_\a\p_\a+(1+s^2(1-2ms)P_s)\p_3+P_s\p_4\ri)
$$
where $-L=2P_s+s^2(1-2ms)P_s^{2}+|\wn P|^2.$ By Lemma \ref{l-induced-connection},  \eqref{e-connection}   and
\bee
\ol A_{ij}=-\la \ol\nabla_{e_i}e_j,\nu\ra_{\ol g},
 \eee
 then $\ol A_{ij}$ are as in the lemma. The last result follows from the relation on the second fundamental forms between two conformal metrics. For example,
 \be
\begin{split}
 \ol A_{33}=
&-\la \nabla_{e_3}e_3,\nu\ra\\
=& L^{-\frac12} \bigg\{-s^2(1-2ms)P_s\lf(-P_{ss}+sP_s^2-3ms^2P_s^2\ri)
\\
&+\lf(-P_{ss}+sP_s^2-3ms^2P_s^2\ri)(1+s^2P_s-2ms^3P_s)\\
& + P_s\lf(   2P_s s -6 mP_s s^2+ P_s^2 s^3-5P_s^2ms^4+6P_s^2m^2s^5\ri)\bigg\}\\
=& L^{-\frac12}\bigg\{6m^2P_s^3s^5- 5mP_s^3s^4  +P_s^3s^3- 9m P_s^2s^2  +3P_s^2s-P_{ss}\bigg\}\\
=&L^{-\frac12}(-P_{ss}+3P_s^2s-9m P_s s^2)+O(s^3).
\end{split}
\ee

 \end{proof}
 \begin{lma}\label{l-decay} With the above notations, we have $|\mathring A|_G=O(s^n)$, $1\le n\le 4$ if and only if
 $$
 s(\ol A_{ij}+B\ol g_{ij})- \ol g_{ij}=O(s^n).
 $$
Here $\ol A_{ij}, \ol g_{ij}$ are considered as functions.
 \end{lma}
 \begin{proof}
 Let $\varepsilon_i$ be an orthonormal basis on $\Sigma$ with respect to $\ol g$ and let
 $$
 e_i=\sum_k a_{ik}\varepsilon_k.
 $$
 Then
 $$
 |a_{ik}|=|\la e_i,\varepsilon_k\ra|\le C
 $$
 near $s=0$, because by Lemma \ref{l-induced-connection} at $s=0$
 \be\label{e-g-matrix}
 (\ol g_{ij} ) =\left(
           \begin{array}{ccc}
             1 & 0 & - P_1\\
             0 & 1 & - P_2 \\
             - P_1 &  - P_2 &    1+|\wn P|^2 \\
           \end{array}
         \right).
 \ee
 Hence $\det (\ol g_{ij})=1$ at $s=0$.

 On the other hand,
 $$
 \ol g_{ij}=a_{ik}\la \varepsilon_k , e_j\ra_{\ol g}
 $$
 So
 $$
 (a_{ik})^{-1}=\lf(\la \varepsilon_k , e_j\ra_{\ol g}\ri)(\ol g_{ij})^{-1}.
 $$
 We also have $|a^{ik}|\le C$ near $s=0$, where $(a^{ik})=(a_{ik})^{-1}$ so that $a^{ik}a_{kj}=a_{ik}a^{kj}=\delta_{ij}$.

 Suppose $|\mathring A|=O(s^n)$. Then
 \bee
 \begin{split}
 s(\ol A_{ij}+B\ol g_{ij})- \ol g_{ij} =&s^2   \lf(s^{-1}(\ol A_{ij}+B\ol g_{ij})-s^{-2}\ol g_{ij}\ri)\\
 =&s^2\lf(A_{ij}-G_{ij}\ri)\\
 =&s^2(\sum_{k=1}^3a_{ik}a_{jl})\lf(A(\varepsilon_k,\varepsilon_l)-G( \varepsilon_k,\varepsilon_l)\ri)\\
 =&(\sum_{k=1}^3a_{ik}a_{jl})\lf(A(s\varepsilon_k,s\varepsilon_l)-G( s\varepsilon_k,s\varepsilon_l)\ri)\\
 =&O(s^n)
 \end{split}
 \eee
 because $\{s\varepsilon_k\}_{k=1}^3$ is an orthonormal basis with respect to $G$ and $|a_{ik}|$ are uniformly bounded.

 The converse can be proved similarly.

 \end{proof}

Now we are ready to prove Theorem \ref{t-mathringA}.

\begin{proof}[Proof of Theorem \ref{t-mathringA}]
By Lemma \ref{l-2nd}, we see that $\ol A_{ij}$ are bounded and $sB-1=O(s)$ because $L=1$ at $s=0$. Hence by Lemma \ref{l-decay}, we have $|\mathring A|=O(s)$.

It is obvious that (a)(iii) implies (a)(i). To prove (a)(i) implies (a)(ii), let $\ol A_{ij}, A_{ij}$ as in Lemma \ref{l-2nd}. As in the lemma, let $\mathbf{y}_0\in \mS^2$ and choose a normal coordinates near this point. Then
at $s=0$, using Corollary \ref{c-smooth-null}, Proposition \ref{p-L=1} and Lemmas \ref{l-2nd}, \ref{l-decay}, we have
 \bee
 \begin{split}
 0=&\frac{\p}{\p s}\lf[s\ol A_{\a\b}+\lf(L^{-\frac12}\lf(1+s^2(1-2ms)P_s\ri)-1\ri)\ol g_{\a\b}\ri] \\
 =&\ol A_{\a\b}-\frac12 L_s \ol g_{\a\b} \\
 =&-P_{\a\b}+\frac12 (\wt \Delta P)\delta_{\a\b}
 \end{split}
 \eee
 at  $ \mathbf{y}_0{\in \mS^2}$.  Hence at $ \mathbf{y}_0$ and $s=0$,     $\wn^2P=\frac12 (\wt\Delta P)\sigma$, where $\sigma$ is the standard metric on $\mathbb{S}^2$. Since $\mathbf{y}_0$ is an arbitrary point on $\mS^2$, we conclude that at $s=0$, $\wn^2P=\frac12 (\wt\Delta P)\sigma$ on $\mS^2$. It is well-known that $P$ must be a linear combination of the first four eigenfunctions of the Laplacian of $\mS^2$. In fact,
 by the Bochner formula

\bee
\wt\Delta |\wt\nabla P|^2=2|\wt\nabla^2 P|^2+2\la \wt\nabla \wt\Delta P, \wt\nabla P\ra+2|\wt\nabla P|^2,
\eee
because $\mS^2$ has constant curvature 1. Hence we have
 \be\label{e-Bochner}
 \begin{split}
 0=&\int_{\mS^2}(|\wt\nabla^2 P|^2+\la \wt\nabla \wt\Delta P, \wt\nabla P\ra+ |\wt\nabla P|^2)\\
 =&\int_{\mS^2}\frac12 (\wt\Delta P)^2-(\wt\Delta P)^2+|\wt\nabla P|^2\\
 =&\int_{\mS^2}-\frac12 (\wt\Delta P)^2-P\wt\Delta P.
 \end{split}
 \ee
 Let $\phi_i$, $i=0,1, \dots$ be the eigenfunctions of $(- \wt\Delta)$ with eigenvalues $\lambda_i$ normalized so that
 $$
 \int_{\mathbb{S}^2}\phi_i\phi_j dA_\sigma=\delta_{ij}
 $$
 where $dA_\sigma$ is the area element of $\mathbb{S}^2$ with respect to the standard metric. Moreover, $\lambda_0=0$ so that $\phi_0$ is constant; $\lambda_1=\lambda_2=\lambda_3=2$ and $\lambda_i>2$ for $i>3$. Then
 $$
 P= \sum_{i=0}^\infty a_i \phi_i
 $$
 in $L^2(\mathbb{S})$, where $a_i=\int_{\mathbb{S}^2}   P \phi_i  dA_\sigma$. Similarly
 $$
   \wt \Delta  P = \sum_{i=0}^\infty b_i\phi_i
 $$
 where
 $$
 b_i=\int_{\mathbb{S}^2} \phi_i\wt
 \Delta P dA_\sigma=\int_{\mathbb{S}^2}  P \wt\Delta \phi_i  dA_\sigma=-\lambda_ia_i.
 $$
 By \eqref{e-Bochner},
 \bee
 \begin{split}
 0=&-\frac12\sum_{i=1}^\infty \lambda_i^2a_i^2+\sum_{i=1}^\infty \lambda_i a_i^2\\
 =&\sum_{i=1}^\infty \lambda_i a_i^2(1-\frac12\lambda_i).
 \end{split}
 \eee
 Since $\lambda_i\ge 2$ for $i\ge 1$ and $\lambda_i>2$ for $i\ge 4$, we have $a_i=0$ for $i\ge 4$. Therefore, $P=a_0+\sum_{i=1}^3a_i\phi_i$.

In the sequel of the proof of this theorem, by adding a constant to $P$, we may assume $P$ is an eigenfunction with eigenvalue 2 or is zero at $s=0$. In any case, we may assume that $P=\mathbf{a}\cdot \mathbf{x}$, where $\mathbf{a}$ is a constant vector (which might be zero) and $\mathbf{x}$ is the position vector of points in $\mathbb{S}^2$.

Next we want to prove (a)(ii) implies (a)(iii). We will give two proofs.

 {\bf First proof}: We have

 \be\label{e-first-eigenfunction}
\wt\Delta P=-2P;\ \ |\wn P|^2=|\mathbf{a}|^2-P^2.
 \ee
Combining with   Corollary \ref{c-smooth-null}, we have at $s=0$
\be\label{e-PLs}
\left\{
  \begin{array}{ll}
L=1\\
    P_s=-\frac12(1+|\wn P|^2)\\
L_s=2P\\
P_{ss}=-P(1+|\wn P|^2)\\
L_{ss}
=6P^2-2(1+|\wn P|^2)\\
P_{sss}= 3P_s^2+6P^2P_s.
 \end{array}
\right.
\ee
By Lemma \ref{l-decay}, it is sufficient to prove that at $s=0$,
\be\label{e-os3}
\left\{
  \begin{array}{ll}
   s\ol A_{ij}+\lf(L^{-\frac12}\lf(1+s^2(1-2ms)P_s\ri)-1\ri)\ol g_{ij}=0;\\
    \frac{\p}{\p s}\lf[s\ol A_{ij}+\lf(L^{-\frac12}\lf(1+s^2(1-2ms)P_s\ri)-1\ri)\ol g_{ij}\ri]=0\\
    \frac{\p^2}{\p s^2}\lf[s\ol A_{ij}+\lf(L^{-\frac12}\lf(1+s^2(1-2ms)P_s\ri)-1\ri)\ol g_{ij}\ri]=0
  \end{array}
\right.
\ee
By the first part of the theorem,   the first relation is true. As in the above, at $s=0$, we have
\bee
\frac{\p}{\p s}\lf[s\ol A_{ij}+\lf(L^{-\frac12}\lf(1+s^2(1-2ms)P_s\ri)-1\ri)\ol g_{ij}\ri]=\ol A_{ij}+(\frac12\wt\Delta P)\ol g_{ij}.
\eee
By the proof above and the fact that $\wn^2P=(\frac12\wt\Delta P)\sigma$,  we conclude that at $s=0$
\bee
\frac{\p}{\p s}\lf[s\ol A_{\a\b}+\lf(L^{-\frac12}\lf(1+s^2(1-2ms)P_s\ri)-1\ri)\ol g_{\a\b}\ri]=0.
\eee
On the other hand, by Lemma \ref{l-2nd}, \eqref{e-first-eigenfunction} and \eqref{e-PLs},  at $s=0$ we have:
\bee
\begin{split}
\ol A_{33}+(\frac12\wt\Delta P)\ol g_{33}=&-P_{ss}-P_s\wt\Delta P\\
=& \frac12\la \wn P^2,\wn P\ra+P- (1+|\wn P|^2)P\\
=&0;
\end{split}
\eee
and
\bee
\begin{split}
\ol A_{3\a}+(\frac12\wt\Delta P)\ol g_{3\a}=&-P_{\a s}+P P_{\a } \\
=&\frac12(|\wn P|^2)_\a+PP_\a\\
=&0.
\end{split}
\eee
Hence the second equation in \eqref{e-os3} is true.

Next, by Lemma \ref{l-2nd}, \eqref{e-first-eigenfunction} and \eqref{e-PLs},  at $s=0$ we have:
\bee
\begin{split}
&\frac{\p^2}{\p s^2}\lf[s\ol A_{ij}+\lf(L^{-\frac12}\lf(1+s^2(1-2ms)P_s\ri)-1\ri)\ol g_{ij}\ri]\\
=&2\frac{\p}{\p s}\ol A_{ij}+( \frac34L_s^2-\frac12 L_{ss}  +2P_s)\ol g_{ij}-L_s\frac{\p}{\p s}\ol g_{ij}\\
=&2\frac{\p}{\p s}\ol A_{ij}+(\frac34L_s^2-\frac12 L_{ss}+2P_s)\ol g_{ij}-2P\frac{\p}{\p s}\ol g_{ij}\\
=&2\frac{\p}{\p s}\ol A_{ij} -2P\frac{\p}{\p s}\ol g_{ij}.
\end{split}
\eee
 On the other hand,
\bee
\begin{split}
 2\frac{\p}{\p s}\ol A_{\a\b} -2P\frac{\p}{\p s}\ol g_{\a\b}
=&2P P_{\a\b}-2P_{\a\b s}+2P_\a P_\b   \\
=&0
\end{split}
\eee
where we have used the fact that $P_s=-\frac12(1+|\wn P|^2)=-\frac12(1+|\mathbf{a}|^2-P^2)$, so that $2P_{\a\b s}= (P^2)_{\a\b}$ and that $L_s=2P$.
\bee
\begin{split}
 2\frac{\p}{\p s}\ol A_{\a 3} -2P\frac{\p}{\p s}\ol g_{\a 3}
=&2PP_{\a s}-2P_{\a ss}+4P_\a P_s+2PP_{\a s}  \\
=&0
\end{split}
\eee
where we have used the fact that $P_{ss}=2PP_s$.
\bee
\begin{split}
 2\frac{\p}{\p s}\ol A_{33} -2P\frac{\p}{\p s}\ol g_{3 3}
=&2PP_{s s}-2P_{s ss}+6 P_s^2+4PP_{s s}  \\
=&6PP_{ss}-2P_{sss}+6P_s^2\\
=&12P^2P_s-6P_s^2-12P^2P_s+6P_s^2\\
=&0.
\end{split}
\eee
Hence the third equality in \eqref{e-os3} is also true.

\vskip  .2cm

{\bf Second proof}: Consider the constant mean curvature surface in the Minkowski spacetime which is the entire graph of the function:

\be\label{u0}
  u^0(x)=\lf( 1+|\mathbf{x}+\mathbf{a}|^2\ri)^\frac12.
\ee
Let $ P^0=   r -u^0 $ where $r=|x|$. Then at the future null boundary, $   P^0=\mathbf{a}\cdot \mathbf{x}$ on $\mathbb{S}^2$. Let $F^0=v^0+P^0$,   $v^0=t-r$ and let $L^0=-\la \nabla^0 F ^0, \nabla^0 F  ^0 \ra$ where $\nabla^0  F  ^0$ is the gradient of $F ^0$ with respect to the Minkowski metric. Let $ \ol A^0, \ol g^0$ be the second fundamental form with respect to the future pointing unit normal and the induced metric on $\Sigma_0$. Let $s=r^{-1}$. Then at a point $(\mathbf{y},s)$ with $\mathbf{y}\in \mathbb{S}^2$, we have
$$
L^0=L+O(s^2)
$$
by the fact that $P^0=P$ at $s=0$ and Theorem \ref{t-smooth-future}. Also $\ol A^0=\ol A+O(s^2)$; $\ol g^0=\ol g+O(s^2)$ with respect to the corresponding basis $e_i^0, e_i$ defined as before. Hence, for example, we have
 $$
 \frac{\p^2}{\p s^2}(s\ol A_{ij})=\frac{\p^2}{\p s^2}(s\ol A^0_{ij})
 $$
 at $s=0$. Since the equalities in \eqref{e-os3} are true for $\ol A^0, \ol g^0$, it is easy to see that the same will  hold for $\ol A, \ol g$ by Lemma \ref{l-scalar} because the surface given by $u^0$ is the hyperbolic space of constant mean curvature -1.  This completes the proof of (a)(ii)$\Rightarrow$(a)(iii).

To prove (b), by (a) by adding a constant to $P$ we have $P=\mathbf{a}\cdot \mathbf{x}$ for some constant vector $\mathbf{a}$ where $\mathbf{x}$ is the position vector of $\mS^2$ in $\R^3$. Let $P^0,P^0_s,P^0_{\a}$ etc. denote the corresponding quantities for the graph of $u^0=(1+|\mathbf{x}+\mathbf{a}|^2)^\frac12$ in Minkowski spacetime. Let
\bee
P_{ssss}(0)=b,\ \  P^0_{ssss}(0)=c,
\eee
where $b, c$ are smooth functions on $\mS^2$. By Theorem \ref{t-smooth-future}, in the setting of Lemma \ref{l-2nd},
\be
\left\{
  \begin{array}{ll}
    P-P^0=\frac1{4!}(b-c)s^4+O(s^5)\\
P_s-P^0_s=\frac1{3!}(b-c)s^3+O(s^4)\\
P_{ss}-P^0_{ss}=\frac1{2!}(b-c)s^2+O(s^3)\\
P_{s\a}-P^0_{s\a}=O(s^3)\\
P_\a- P^0_{\a}  =O(s^4)\\
P_{\a\b}-P^0_{\a\b}=O(s^4)
  \end{array}
\right.
\ee
etc. Also from Lemma \ref{l-2nd},
$$
L=-(2P_s+s^2(1-2ms)P_s^2+|\wn P|^2), sB= L^{-\frac12}(1+s^2(1-2ms)P_s),
$$
then
\bee
\begin{split}
L-L^0
&=- (\frac13 (b-c)-2mP_s^2(0) )s^3+O(s^4).
\end{split}
\eee
\be\label{sB-sB0}
\begin{split}
s(B-B^0)=&L^{-\frac12}(1+s^2(1-2ms)P_s)-(L^0)^{-\frac12}(1+s^2P_s^0)\\
=&\lf(-2mP_s(0)+  \frac16 (b-c)- mP_s^2(0)\ri)s^3+O(s^4)
\end{split}
\ee

By Lemma \ref{l-decay}, $|\mathring A|=O(s^4)$ if and only if
$$
s(\ol A_{ij}+Bg_{ij})-\ol g_{ij}=O(s^4).
$$
Since $\mathring A^0=0$, if $|\mathring A|=O(s^4)$ then
\bee
s(\ol A_{ij}-\ol A^0_{ij}+B\ol g_{ij}-B^0\ol g^0_{ij})-(\ol g_{ij}-\ol g^0_{ij})=O(s^4).
\eee
Using Lemma \ref{l-induced-connection}, $\ol g_{ij}-\ol g^0_{ij}=O(s^3)$, so
\be
\begin{split}
(sB-1)\ol g_{ij}-(sB^0-1)\ol g^0_{ij}=&(sB-sB^0)\ol g_{ij}+(sB^0-1)(\ol g_{ij}-\ol g^0_{ij})\\
=&(sB-sB^0)\ol g_{ij}+O(s^4)
\end{split}
\ee
Hence $|\mathring A|_G=O(s^4)$ if and only if
\be\label{AsB}
s(\ol A_{ij}-\ol A^0_{ij})+(sB-sB^0)\ol g_{ij}=O(s^4)
\ee
By   Lemma \ref{l-2nd},

\bee
\begin{split}
&s(\ol A_{\a\b}-\ol A^0_{\a\b})\\
=& s(L^0)^{-\frac12}
\lf( -3ms^2 P_\a P_\b  +sP_\a P_\b-P_{\a\b}  - sP_\a^0 P_\b^0+P^0_{\a\b}\ri)+O(s^4)\\
=&3ms^3(P_\a P_\b)(0)+O(s^4)
\end{split}
\eee
Together with \eqref{sB-sB0} and \eqref{AsB}, if $|\mathring A|_G=O(s^4)$, then
\bee
 3m (P_\a P_\b)(0)s^3+ \lf(-2mP_s(0)+  \frac16 (b-c)- mP_s^2(0)\ri)\delta_{\a\b}s^3=O(s^4)
\eee
So we must have $P_\a P_\b=0$ for $\a\neq \b$. If $\wn P\neq 0$, at such a point, we can choose an orthonormal frame so that $\wn P/|\wn P|=\frac1{\sqrt 2}(e_1+e_2)$. Then $P_1 P_2\neq 0$, which leads a contradiction. So we derive that $\mathbf{a}=\mathbf{0}$, that is $P\equiv 0$ at $s=0$. Moreover, by $P_s(0)=-\frac12$ and $c=0$. Hence $b=P_{ssss}(0)$ is such that
\bee
-2mP_s(0)+  \frac16 (b-c)- mP_s^2(0)=0,
\eee
so
\bee
m+  \frac16 b-\frac34m=0,
\eee
that is $b=-\frac{9m}2.$ Thus (b)(i)$\Rightarrow$(b)(ii).

On the other hand, if $P\equiv 0$ and $b=-\frac{9m}2$ at $s=0$, then $s(\ol A_{\a\b}-\ol A^0_{\a\b})=0$ and $s(B-B^0)=0$.
In case $i=j=3$, by  Lemma \ref{l-2nd}, we have
\be
\begin{split}
&s(\ol A_{33}-\ol A^0_{33})
\\=& sL^{-\frac12}\lf(-9m P_s^2s^2+3P_s^2s-P_{ss}\ri) -
s(L^0)^{-\frac12}\lf(  3(P_s^0)^2s-P^0_{ss}\ri)+O(s^4)\\
=&s(L^0)^{-\frac12}\lf( - 3(P_s^0)^2s + P^0_{ss}-9m P_s^2s^2 +  3P_s^2s - P_{ss}\ri)+O(s^4)\\
=&s(L^0)^{-\frac12}(0)\lf(-\frac1{2!}(b-c)s^2 -9m P_s^2(0)s^2  \ri)+O(s^4)\\
=&s(-\frac12 b-\frac{9m}4)+O(s^4)=O(s^4)
\end{split}
\ee

Similarly, in case $i=3,j=\a$, direct computation gives
\bee
\begin{split}
\ol A_{3\a}
=&  L^{-\frac12} (-P_{s\a}+ 2P_sP_\a s-6mP_sP_{\a }s^2)  + O(s^{3}).
\end{split}
\eee
Then
\bee
\begin{split}
&s(\ol A_{3\a}-\ol A^0_{3\a})\\
=& s(L^0)^{-\frac12}
\lf( -P_{s\a}+ 2P_sP_\a s-6mP_sP_{\a }s^2+ P^{0}_{s\a}- 2P^{0}_sP^{0}_\a s\ri)+O(s^4)\\
=&-6m(P_sP_{\a })(0)s^3+O(s^4)=O(s^4).
\end{split}
\eee
Thus we derive \eqref{AsB}. This completes the proof of (b)(ii)$\Rightarrow$(b)(i).

\end{proof}

\subsection{Applications}\label{ss-applications}

By the works of Treibergs \cite{Treibergs} and Choi-Treibergs \cite{ChoiTreibergs1990}, it is known that there are no Bernstein results for entire graphs of nonzero spacelike CMC surfaces in Minkowski spacetime.  On the other hand, it was proved by Choquet-Bruhat \cite{Choquet-Bruhat1976} that if the entire spacelike graph with constant mean curvature 1 given by $u$ with $u-r=o(r^{-1})$, then $u=(1+r^2)^\frac12$ where $r=|x|$.  In \cite{Goddard1977} Goddard  proved that a spacelike CMC surface in the Minkowski spacetime which satisfies infinitessimal constant mean curvature perturbation, then the surface is a hyperboloid. We have the following uniqueness result in terms of the decay rate of scalar curvature.
\begin{thm}\label{uniqueness1}
Let $\Sigma$ be a complete spacelike hypersurface with constant mean curvature $H=1$ in the Minkowski spacetime $\mathbb{R}^{3,1}$, which is the entire graph of $u$. Suppose $P(u):={r-u}$ is smooth up to $s=0$ in the setting in \S\ref{ss-CMC-future-2} (with $m=0$), and suppose the scalar curvature of $\mathcal{S}$ of  $\Sigma$  satisfies $|\mathcal{S}+6|=O(s^4)$. Then
 $$
 u(x)=\lf(1+|\mathbf{x}+\mathbf{a}|^2\ri)^\frac12+b
 $$
 for some constant vector $\mathbf{a}\in \R^3$ and a constant $b$. In particular,  	$\Sigma$ is isometric to the hyperbolic space $ \mathbb{H}^3$.
\end{thm}

\begin{proof} By Lemma \ref{l-scalar}, the traceless part $\mathring A$ of the second fundamental form satisfies $|\mathring A|_G=O(s^2)$ where $G$ is the induced metric on $\Sigma$. Let $P=r-u$. By
  Theorem \ref{t-mathringA}, we know that the corresponding $P=a_0+\mathbf{a}\cdot \mathbf{x}$ at $s=0$ where $|\mathbf{x}|=1$.  By a translation in the $t$-direction, we may assume $P=\mathbf{a}\cdot \mathbf{x}$ at $s=0$. Let $u^0=\lf(1+|\mathbf{x}+\mathbf{a}|^2\ri)^\frac12$. Then $P^0=r-u^0$ is smooth up to $s=0$ and $P^0=r-u^0\to \mathbf{a}\cdot \mathbf{x}$ as $s\to 0$, i.e. as $r\to\infty$. Hence $u-u^0=P^0-P$ will  uniformly converge to zero as $r\to\infty$, then by maximum principle, we see that $u=u^0$. This also  implies that 	$\Sigma$ is isometric to the hyperbolic space $\mathbb{H}^3$.
\end{proof}

The second application is the following. In the construction of spacelike constant mean curvature surfaces, it is useful to construct barriers at the null-infinity first, see \cite{Treibergs,AnderssonIriondo1999} for example. By Theorem  \ref{t-smooth-future}, one can construct spacelike surfaces near the future null-infinity with mean curvature being asymptotically constant.
\begin{thm}\label{t-barrier}
  For any constant $H_0>0$ and any smooth function $f$ on $\mS^2$, one can find a spacelike surface $\Sigma$ near $\mathcal{I}^+$, which is a graph of a function $P(\mathbf{y},s)$ in the coordinates ${(}\mathbf{y},s,\bar v)$ so that the mean curvature
  $H$ satisfies $H=H_0+\mathcal{O}(s^3)$ and $P$ satisfies the boundary condition
  $P(\mathbf{y},0)=f(\mathbf{y})$.

\end{thm}
\begin{proof} We only prove the case for $H_0=1$. The general case is similar. Define
\bee
\left\{
  \begin{array}{ll}
    f_1:=-\frac12 (1+|\wn f|^2 );\\
f_2:=\frac12\la\wn (|\wn f|^2), \wn f\ra+\frac12 \wt\Delta f;\\
f_3:=\frac14(\wt\Delta f)^2- \wt\Delta |\wn f|^2+ \la \wn\wt\Delta f,\wn f\ra+ (1+|\wn f|^2 )\\
\ \ \ \ \ \ -\la \wn f_2,\wn f\ra-|\wn f_1|^2-f_1^2.
  \end{array}
\right.
\eee
 Let $$P(\mathbf{y},s)=f(\mathbf{y})+sf_1(\mathbf{y})+\frac1{2!}s^2f_2(\mathbf{y})+\frac1{3!}s^3f_3(\mathbf{y}).$$
Then at $s=0$, $P=f$, $P_s=f_1$, $P_{ss}=f_2$ and $P_{sss}=f_3$. Moreover, $P$ is smooth up to $s=0$ and the graph of $P$ is spacelike if $s$ is small enough. By Corollary \ref{c-smooth-null}, we conclude that the mean curvature of the graph is  $H=H_0+\mathcal{O}(s^3)$.
\end{proof}

From \S\ref{ss-SSCMC}, we know that a spacelike SSCMC surface is asymptotically hyperbolic in the sense of \cite{Wang2001}. Using Corollary \ref{c-smooth-null}, Theorem \ref{t-mathringA} and Lemma \ref{l-scalar} we obtain the following for general spacelike CMC surfaces.

\begin{thm}\label{t-AH-CMC} Let $\Sigma$ be a spacelike CMC surface which is the graph of $P=P(\mathbf{y},s)$ defined near the null-infinity satisfying Assumption \ref{a1}. Namely, $P$ is smooth up to $s=0$ and   $\Sigma$ has constant mean curvature 1.  Suppose $P(\mathbf{y},0)$ is a linear combination of the first four eigenfunctions of the Laplacian on the standard unit sphere $\mS^2$. Then $\Sigma$ is asymptotically hyperbolic in the sense of Chru\'sciel-Herzlich \cite {ChruscielHerzlich}.

\end{thm}
Let us recall the  definition of asymptotically hyperbolic  manifold given in \cite{ChruscielHerzlich}.
 Let $\H^n$  denote   the standard hyperbolic space with  metric $\mathfrak{b}$    given by
\be\label{e-AH-metric}
\mathfrak{b}=d\rho^2+ (\sinh \rho)^2 \sigma
\ee
where $\sigma$ is  the standard metric on $\mathbb{S}^{n-1}$. Choose    finitely many geodesic balls of radius $2$ in  $(\mathbb{S}^{n-1}, \sigma)$,
so that the geodesic balls of  radius $1$ with the same centers cover $\mathbb{S}^{n-1}$.
In each geodesic ball, fix an orthonormal frame  $ \{ f_\a  \}_{1 \le \a  \le n-1}$.
Let
\bee
\ve_0=\p_\rho;\ \  \ve_\alpha =\frac1{\sinh\rho}f_\alpha, \ 1 \le \a  \le n-1,
\eee
where $\{f_\alpha \}_{1\le \alpha \le n-1}$ is a local orthonormal
frame on $\mathbb{S}^{n-1}$. Then  $\{ \ve_i \}_{0 \le i \le n-1 }$ form a local  orthonormal  frame on $\H^n$. In the following, $i, j,\dots$ range from 0 to $n-1$ and $\a,\b,\dots$ range from 1 to $n-1$. A metric $g$ defined outside some compact set of $\H^n$ is said to be {\it asymptotically hyperbolic}, if the metric components  $ g_{ij} =  g(\ve_i,\ve_j)$,  $0 \le i, j \le n-1$,  satisfy
\be\label{e-AH}
|g_{ij}-\delta_{ij}|=O(e^{-\tau \rho}),\ |\ve_k(g_{ij})|=O(e^{-\tau \rho}), \ |\ve_k(\ve_l(g_{ij}))|=(e^{-\tau \rho})
\ee
for some $\tau>\frac n2$, such that the scalar curvature $\cS$ satisfies that $\mathcal{S} +n(n-1)$ is in
$L^1( e^\rho dv_\mathfrak{b})$ near infinity,  where  $dv_\mathfrak{b}$ is the volume element of $\mathfrak{b}$. By extending the metric $g$ smoothly inside the compact set, we may assume that $g$ is defined on $\H^n$. Note that if $g$ is strongly asymptotically hyperbolic, then it is asymptotically hyperbolic in the above sense.

First, we need   to find   intrinsic conditions for a manifold being asymptotically hyperbolic.

\begin{lma}\label{l-AH-intrinsic}
Let $g$ be a smooth metric on $\H^n$ and let $\theta=g-\mathfrak{b}$. Then $g$ satisfies \eqref{e-AH} for some $\tau>0$, if $|\theta|_\mathfrak{b}$, $|\nabla \theta|_\mathfrak{b}$ and $|\nabla^2\theta|_\mathfrak{b}$ are of order $e^{-\tau d}$ where $d$ is the distance from a fixed point with respect to $\mathfrak{b}$, and $\nabla$ is the covariant derivative with respect to $\mathfrak{b}$.
\end{lma}
\begin{proof} In the metric of the form \eqref{e-AH-metric}, then $|d-\rho|$ is bounded. Hence a function which is $O(e^{-\tau d})$ if and only if it is $O(e^{-\tau \rho})$. However, $d$ is more intrinsic. Define
$ \{ \Gamma_{ij}^k \} $   by:
\be \label{e-def-gamma-t-gamma}
\nabla_{\ve_i}\ve_j=\Gamma_{ij}^k \ve_k.
\ee
Direct computations give:
\be \label{e-gamma-tgamma}
\begin{split}
 \Gamma_{ij}^k
=&\ \frac12     \lf[- \mathfrak{b}([\ve_i,\ve_k ], \ve_j)-\mathfrak{b}([\ve_j,\ve_k ],\ve_i)+ \mathfrak{b}([\ve_i,\ve_j],\ve_k)\ri] ,
\end{split}
\ee
and
\be \label{e-Lie}
[\ve_0,\ve_\a]=-\frac{\cosh \rho}{\sinh \rho}\ve_\a, \  \ [\ve_\a,\ve_\b]=\frac1{ (\sinh \rho) }\sum_{\gamma=1}^{ n-1} \lambda_{\a\b}^\gamma\ve_\gamma ,
\ee
 where $ \{ \lambda_{\a\b}^\gamma \} $ are smooth functions in the geodesic balls of radius $2$ in $(\mathbb{S}^{n-1}, \sigma)$ so that geodesic balls of radius 1 will cover $\mathbb{S}^{n-1}$. Moreover
 \be\label{e-Gamma-1}
 \left\{
   \begin{array}{ll}
     \Gamma_{00}^k=\Gamma_{0\a}^0=\Gamma_{\a 0}^0=\Gamma_{0\a}^\b=0;
\Gamma_{\a 0}^\b=\coth\rho \, \delta_{\a\b};\\
\Gamma_{\a\b}^0=-\coth\rho\, \delta_{\a\b};\\
\Gamma_{\a\b}^\gamma=\frac1{2\sinh \rho}\lf(-\lambda_{\a\gamma}^\b-\lambda_{\b\gamma}^\a+\lambda_{\a\b}^\gamma\ri).
   \end{array}
 \right.
 \ee
It is easy to see that:
\be\label{e-Gamma-2}
\ve_l(\Gamma_{ij}^k)=O(e^{-\rho}).
\ee
  Suppose $g$ is   metric defined on $\H^n$ such that $|\theta|_{\mathfrak{b}}=|g-\mathfrak{b}|_{\mathfrak{b}}=O(e^{-\tau d})$, then one can see that
$$
|g_{ij}-\delta_{ij}|=O(e^{-\tau \rho}).
$$
because $|d-\rho|$ is bounded. In addition, suppose $|\nabla \theta |_{\mathfrak{b}}=O(e^{-\tau d})=O(e^{-\tau \rho})$.  Then
\bee
\theta_{ij;k}=\ve_k(\theta_{ij})-\Gamma_{ki}^l\theta_{lj}-\Gamma_{kj}^l\theta_{il}
\eee
Since $\Gamma_{ij}^k$ are uniformly bounded, and $|\theta_{ij}|, |\theta_{ij;k}|=O(e^{-\tau \rho})$, we conclude that $\ve_k(g_{ij})=\ve_k(\theta_{ij})=O(e^{-\tau \rho})$.
Suppose we also have $|\nabla^2\theta|_{\mathfrak{b}}=O(e^{-\tau d})=O(e^{-\tau \rho})$, we prove $\ve_l(\ve_k(g_{ij}))=O(e^{-\tau \rho})$. By
\bee
\theta_{ij;k}=\ve_k(\theta_{ij})-\Gamma_{ki}^l\theta_{lj}-\Gamma_{kj}^l\theta_{il},
\eee
then
\bee
\theta_{ij;kl}= \ve_l(\theta_{ij;k})-\Gamma_{li}^m\theta_{mj;k}-\Gamma_{lj}^m\theta_{im;k}-\Gamma_{lk}^{m}\theta_{ij;m}.
\eee
 Again, as $\Gamma_{ij}^k$ are uniformly bounded and $|\nabla \theta|_{\mathfrak{b}}=O(e^{-\tau \rho})$, the left hand side and the last three terms are of order $e^{-\tau \rho}$. On the other hand,
\bee
\begin{split}
\ve_l(\theta_{ij;k})=&\ve_l\lf(\ve_k(\theta_{ij})-\Gamma_{ki}^m\theta_{mj}
-\Gamma_{kj}^m\theta_{im}\ri) .
\end{split}
\eee
Since $\ve_l(\ve_k(\theta_{ij}))=\ve_l(\ve_k(g_{ij}))$ and
\bee
\ve_l(\Gamma_{ki}^m\theta_{mj})=\ve_l(\Gamma_{ki}^m)\theta_{mj}+\Gamma_{ki}^m\ve_l(\theta_{mj})
=O(e^{-\tau \rho})
\eee
etc., by \eqref{e-Gamma-2}, and the fact that $|\theta|_{\mathfrak{b}}, |\nabla\theta|_{\mathfrak{b}}$ are of order $e^{-\tau\rho}$, we have $\ve_l(\ve_k(g_{ij}))=O(e^{-\tau\rho})$. This completes the proof of the lemma.

\end{proof}

We are ready to prove Theorem \ref{t-AH-CMC}.
\begin{proof}[Proof of Theorem \ref{t-AH-CMC}] We may assume that $P|_{s=0}=\mathbf{a}\cdot \mathbf{x}$, where $\mathbf{x}$ is the position vector of a point on $\mS^2$. Consider the corresponding CMC surface $\Sigma^0$ in the Minkowski spacetime given by the graph of
$$
u^0=\lf(1+|\mathbf{x}+\mathbf{a}|^2\ri)^\frac12.
$$
Let $P^0=r-u^0$. Then the boundary value of $P^0$ is the same as that of $P$. In fact, in the $\mathbf{y}, s$ coordinates, $P-P^0=O(s^4)$ by Corollary \ref{c-smooth-null}.

Then $\Sigma_0$ can be considered to be parametrized as: $(\mathbf{y}, s)\to   (\mathbf{y}, s, -P^0(\mathbf{y},s))$. It is similar for $\Sigma$ defined in the same domain. The metrics for $\Sigma, \Sigma_0$ are given by

\bee
\begin{split}
g
=&s^{-2}\lf(\sigma_{AB}dy^Ady^B-2P_s ds^2-2P_A dy^Ads\ri)\\
&-(1-2ms)\lf(P_AP_B dy^A  dy^B+P_s^2 ds^2+2P_AP_s dy^Ady^s\ri)\\
=:&s^{-2}\ol g.
\end{split}
\eee
and
\bee
\begin{split}
\mathfrak{b}=&s^{-2} \lf(\sigma_{AB}dy^Ady^B-2P_s^0 ds^2-2P_A^0 dy^Ads\ri)\\&-\lf(P_A^0P_B^0 dy^A  dy^B+(P_s^0)^2 ds^2+2P_A^0P_s^0 dy^Ady^s\ri)\\
=:&s^{-2}\ol {\mathfrak{b}}.
\end{split}
\eee
Note that $(\Sigma_0,\mathfrak{b})$ is isometric to $\H^3$.
Now $P-P^0=O(s^4), P_A-P_A^0=O(s^4)$. Hence $\ol g$ and $\ol {\mathfrak{b}}$ are uniformly equivalent. So $g$ and $\mathfrak{b} $ are uniformly equivalent. Consider the coordinate frame $e_1=\p_{y^1}, e_2=\p_{y^2}, e_3=\p_s$. Then $(\ol{\mathfrak{b}}(e_i,e_j))$ and its inverse are uniformly bounded near $s=0$. Let $\ol \theta:=\ol g-\ol{\mathfrak{b}}$ and $\theta:=  g- \mathfrak{b} =s^{-2}\ol \theta.$ Observe that if $T$  is a $(0,k)$ tensor, then
$$
|T|_{\mathfrak{b}}=s^{-k}|T|_{\ol{\mathfrak{b}}}.
$$
By the estimates on $P-P^0$, we have
\be\label{e-theta-1}
|\theta|_{\mathfrak{b}}=|\ol\theta|_{\ol{\mathfrak{b}}}=O(s^3).
\ee
Since $P^0$ is smooth up to $s=0$,  one can also check that
$$
|\ol\nabla\, \ol\theta|_{\ol{\mathfrak{b}}}=O(s^2), \ \ |\ol\nabla^2\,\ol\theta|_{\ol{\mathfrak{b}}}=O(s).
$$
where $\ol \nabla$ is the covariant derivative with respect to $\ol{\mathfrak{b}}.$ Hence we have
$$
|\ol\nabla\,  \theta|_{\ol{\mathfrak{b}}}=O(1), \ \ |\ol\nabla^2  \theta|_{\ol{\mathfrak{b}}}=O(s^{-1}).
$$
From this and the fact that $\Gamma-\ol \Gamma=O(s^{3})$, where $\Gamma, \ol \Gamma$ are the Christoffel symbols with respect to $\mathfrak{b}, \ol{\mathfrak{b}}$ respectively in the coordinates $y^1, y^2, s$, one can check that
$$
|\nabla \theta|_{\mathfrak{b}}=O(s^3);\ \ |\nabla^2 \theta|_{\mathfrak{b}}=O(s^3 ).
$$
where $\nabla$ is the covariant derivative with respect to $ \mathfrak{b}$.

On the other hand, by Theorem \ref{t-mathringA} and Lemma \ref{l-scalar}, the scalar curvature $\mathcal{S}$ of $\Sigma$ satisfies $|\mathcal{S}+6|=O(s^6)$.

We claim that $C_{1}e^{-d}\le s\le C_{2}e^{- d}$ for some constants $C_{1}, C_{2} >0$, provided $s$ is small enough, where $d$ is the distance from a fixed point on $\Sigma_0$ with respect to $\mathfrak{b}$. If the claim is true, then one can see that the theorem is true. To prove the claim, let $\mathbf{z}=\mathbf{x}+\mathbf{a}$ be the translation of $\mathbf{x}$ in the Minkowski spacetime. Then $\Sigma_0$ is the graph of $(1+|\mathbf{z}|^2)^\frac12$. It is easy to see that the distance from a suitable fixed point $d$ satisfies
\bee
\begin{split}
\sinh d=&|\mathbf{z}|\\
=&|\mathbf{x}+\mathbf{a}|\\
=&|\mathbf{x}||1+\frac{\mathbf{a}}{|\mathbf{x}|}|\\
=&s^{-1}|1+s\mathbf{a}|.
\end{split}
\eee
 From this, one can see that the claim is true and so the proof of the theorem is completed.
\end{proof}

\section{Behaviors near $T, X$=0}\label{s-blackhole}
Consider the spacelike SSCMC surface given as a graph of $f(r)$ in the region I constructed by Lee-Lee \cite{LL2}, we know that only in the case that the surface passes through $T=X=0$, $f(r)$ will have finite limit as $r\to 2m$. Motivated by this fact,
 in this section, we will study the behavior of spacelike CMC surfaces near the black hole region. We only study  spacelike CMC surfaces near $T=X=0$ in the Kruskal extension and derive some restrictions on the surfaces.

 \subsection{Behaviors of SSCMC surfaces near $T, X$=0}\label{ss-blackhole-sscmc}
 We first recall the behavior of a spacelike SSCMC hypersurface near $T=X=0$ studied  in \cite{LL} and \cite{LL2}.
 Let $t=f(r)$ be the function of  such a SSCMC hypersurface $\Sigma$ in Schwarzschild spacetime with mass $m>0$ and $c_{1}=-8m^{3}H$ under the usual coordinates $(t, r, \by)$. By Proposition 4(b) in \cite{LL2}, we know that $\lim\limits _{r\rightarrow 2m}f(r)$ is finite, and $\frac{d}{dr}f(r)=O(r-2m)^{-\frac12}$, as $r\rightarrow 2m $. Indeed, by Theorem 5 in \cite{LL2}, we know that the $T$-coordinate function $T$ of this $\Sigma$ is $C^\infty$-smooth with respect to $X$ near $T=X=0$. By Proposition \ref{smoothnessinnerbdry} below, we know that its corresponding $t$-coordinate function $f$ is also $C^\infty$-smooth with respect to $\eta=(1-\frac{2m}r)^\frac12$ near $\eta=0$ which is equivalent to $r=2m$.

From the construction in \cite{LL2}, one can see that for this kind of spacelike SSCMC surfaces $f$ is   determined by $f(2m)$. Hence  {\sl    one cannot prescribe the values of the function   $f(2m)$ and $P(0)$ which is the value of $P$ at null-infinity $s=0$ simultaneously.} In fact, by Proposition 2.1 in \cite{LL} ( see  also Proposition 2 in \cite{LL2} ) with $c_1=-8m^3H$ and let
$$
l_1(r)=h^{-\frac12}(r)(Hr-\frac{8m^3H}{r^2}),
$$
here $h=1-\frac{2m}{r}$, and then we have

$$
f(r)=\int^r_{2m} h^{-1}(x)(1+l^2_1(x))^{-\frac12}l_1(x)dx +\bar c_1
$$
therefore, $\bar c_1=f(2m)$ which implies that we cannot prescribe the value of $P$ at $s=0$ for those functions arising from SSCMC hypersurfaces. On the other hand, for a general spacelike CMC hypersurface under coordinates $(t, r, \by)$, $r\geq 2m$, let $u$ be its $t$-coordinate function, if $u(2m)$ and its corresponding $P(0)$ are both constants, note that Schwarzschild spacetime is spherical symmetric, then by maximum principle, we know that $u$ depends only on $r$, thus, we see that one cannot prescribe the values of the functions $u(2m)$ and $P(0)$  simultaneously even for a general spacelike CMC hypersurface.

 \subsection{Behaviors of general CMC surfaces near $T, X$=0 (I)}\label{ss-blackhole-cmc-1}
 We want to discuss the behaviors of a spacelike CMC graph in Region I as the graph approaches to $T=X=0$ in the Kruskal extension under some conditions on the smoothness.
 Consider  a spacelike surface $\Sigma$ in Region I, which is a graph of a function $u$: $t=u(r,\mathbf{y})$ for $ r_1< r< r_2 $ with $2m\leq r_1 <r_2 \leq \infty$ and $\mathbf{y}\in  \mathbb{S}^2$.  In the Kruskal extension, the coordinates are $(T, X,\mathbf{y})$.
 Hence a point $(u(r,\mathbf{y}), r, \mathbf{y})$ on the surface will be of the form
 $(T(r,\mathbf{y}), X(r,\mathbf{y}), \mathbf{y})$ in the Kruskal extension.   Consider the map: $\Phi:(r_1,r_2)\times \mathbb{S}^2\to \R^+\times \mathbb{S}^2$ with $\Phi(r,\by)=(X(r,\by),\by)$. We have the following:

\begin{lma}\label{l-graph} In the above setting, let $\Omega$ be the image of $\Phi$. Then $\Phi$ is a diffeomorphism onto its image $\Omega$. Hence, $\Sigma$ is given by the graph of $T=T(X,\by)$ over $\Omega$.
\end{lma}

\begin{proof}

By \eqref{e-I}, we have
$$
\frac{\partial X}{\partial r}=\exp (\frac{r_*}{4m})\frac{1}{4m}\lf((\frac{r}{r-2m})\cosh \frac{  u}{4m}+ \frac{\partial u}{\partial r}\sinh \frac{  u}{4m}\ri),
$$

Note that $\Sigma$ is spacelike, we have

$$
|\frac{\partial  u}{\partial r}|< \frac{r}{r-2m},
$$
hence  for any $r>2m$ and $\by \in {\mathbb{S}^2}$, we have

$$
\frac{\partial X}{\partial r}>0.
$$
From this, it is easy to see that $\Phi$ is  a local diffeomorphism. On the other hand, if $\Phi(r,\mathbf{y})=\Phi(r',\mathbf{y}')$, then $\by =\by'$ and $X(r,\by)=X(r',\by')$. By the above inequality, we conclude that $r=r'$. This completes the proof of the lemma.

\end{proof}

In the sequel, we make the following

\begin{ass}\label{assumption1}
Let $\Sigma$ be a smooth  and spacelike graph of a bounded function $u$ in the Region I and can be extended to be a smooth spacelike surface near $T=X=0$ so that $T$ is a smooth function of $X, \by$.
\end{ass}
Since $u$ is bounded, $X(r,\by)\to 0$ if and only if $r\to 2m$. In this case, we also have $T(X,\by)\to 0$.

\begin{prop}\label{smoothnessinnerbdry}
Let  $\Sigma$ be a spacelike  graph of a function $u=u(r,\by)$ inside Region I satisfying 	Assumption \ref{assumption1}. Then $u$ is smooth at $\eta=0$ as a function of $\eta, \by$, where $\eta=h^\frac12=\lf(1-\frac{2m}r\ri)^\frac12$.
\end{prop}
\begin{proof} Let $\Phi$ be the diffeomorphism as above. We may assume that $\Phi$ will map $(0, \eta_0)\times \mS^2$ onto $\Omega$ which contains $(0,a)\times \mS^2$ for some $a>0$. Here $\eta=h^\frac12$ as before. Then any function in $(\eta, \by)$ is also a function in $X, \by$ near $\eta=0$.

Let $\phi$ be a smooth function of $X, \by$ near $X=0$, then it is also a smooth function of $\eta, \by$ near $ 0  $ and vice versa. In the following, as a function of $X, \by$, the partial derivative of $\phi$ with respect to $X, \by$ will be denoted by $\phi_X, \phi_A:=\p_{y^A}$ where $y^A$ are local coordinates of $\mS^2$, $A=1, 2$. As a function of $\eta, \by$, the partial derivatives will be denoted by $\p_{\eta}, \p_A$.

Suppose   $\phi$ is a smooth function of $X, \by$ near $X=0$, then
\be\label{e-PhiX}
\begin{split}
 \phi_{X}=&\p_r\,\phi (r_X+r_TT_X)\\
=&\p_r\phi \cdot \lf(\frac{4m\eta^2 X}{X^2-T^2}-\frac{4m\eta^2 T}{X^2-T^2}T_X\ri)\\
=&4m^2r^{-2} (X^{-1}\eta) \frac{1-T_X\frac TX}{1-(\frac TX)^2}\cdot \p_\eta \phi\\
=&G\cdot \p_{\eta}\phi\\
\end{split}
\ee
where
$$
G= 4m^2r^{-2} (X^{-1}\eta) \frac{1-T_X\frac TX}{1-(\frac TX)^2}.
$$
And
\be\label{e-Phi-y}
\begin{split}
 \phi_{A} =&\p_A\phi +\p_r\phi r_TT_A \\
 =&\p_A\phi  -\p_r\phi   \frac{4m\eta^2T}{X^2-T^2} T_A\\
 =&\p_A\phi  -4m^2r^{-2}\p_\eta\phi  ( X^{-1}\eta) \frac{\frac TX}{1-(\frac TX)^2}\cdot T_A\\
 =& \p_A\phi  -KT_A\p_\eta\phi  \\
\end{split}
\ee
$A=1,2$.
where
$$
K=4m^2r^{-2}  ( X^{-1}\eta) \frac{\frac TX}{1-(\frac TX)^2}.
$$

We claim that $G$, $K$ are smooth functions in $X,\by$  near $X=0$ and $G(0,\by)\ge c>0$.
Suppose the claim is true. Let $\phi$ be a smooth function in $X,\by$. We can conclude that $\p_\eta\phi$ is continuous up to $\eta=0$ by \eqref{e-PhiX},  because $X\to 0$ as $\eta\to0$. Moreover, $\p_\eta\phi$ is a smooth function of $  \eta , \by$ near $ \eta =0$.  Using \eqref{e-Phi-y}, we can also conclude that $\p_A\phi$ is continuous up to $\eta=0$, for $A=1, 2$ and $\p_A\phi$ is a smooth function in $  \eta ,\by$ near $ \eta =0$. Inductively, we can conclude that $\phi$ is a smooth function in $\eta,\by$ up to $\eta=0$.

Apply the argument to the function:
$$
e^{\frac{u}{2m}}=\frac{X+T}{X-T},
$$
we can conclude that $u$ is smooth as a function in $\eta,\by$ at $\eta=0$, provided it is a smooth function in $X, \by$. This is indeed the case, because
\be\label{e-expansion}
\begin{split}
\frac{X+T}{X-T}=&\frac{X+\sum_{k=1}^l\frac{X^k}{k!}T_k+R_{l}} {X-\sum_{k=1}^l\frac{X^k}{k!}T_k+R_l}\\
=&\frac{1+T_1+\sum_{k=2}^l\frac{X^{k-1}}{k!}T_k+
R_{l-1}} {1-T_1 -  \sum_{  k =2}^{l}\frac{X^{k-1}}{k!}T_k+ R_{l-1}}
\end{split}
\ee
where $R_l=O(X^{l+1})$ so that $(R_l)_X=O(X^k), (R_l)_{XX}=O(X^{k-1})$ etc., and $(R_l)_A=O(X^{ l +1})$ etc. Here $T_1=T_X(0,\by)$, $T_2=T_{XX}(0,\by)$ etc., which are smooth functions on $\mS^2$. By
Assumption \ref{assumption1}, near $X=0$, we have $|T_1|<1$.

It remains to prove the claim.  We first consider $X^{-1}\eta$. We have
\bee
\sqrt{\frac{X+T}{X-T}} +\sqrt{\frac{X-T}{X+T}}=\frac{2X}{\sqrt{X^2- T ^2}}=2X\lf((\frac r{2m}-1)\exp(\frac r{2m})\ri)^{-\frac12}.
\eee
So

\bee
X^{-1}\eta=2(\frac{  2 m }{r})^\frac12 \exp(-\frac r{4m}) \left(\sqrt{\frac{X+T}{X-T}} +\sqrt{\frac{X-T}{X+T}}  \right)^{-1}
\eee
By \eqref{e-expansion} and the fact that $|T_1|<1$, we conclude that
$$
X^{-1}\eta(\frac{  2 m }{r})^\frac12 \exp(-\frac r{4m})
$$
is a smooth function in $X,\by$. Using the fact that $ \frac{  2m }{r} =1-\eta^2$,  let
$$
 \psi(\eta):=\eta(\frac{ 2 m }{r})^\frac12 \exp(-\frac r{4m})=\eta\lf(1-\eta^2\ri)^\frac12 \exp\lf(-\frac12 (1-\eta^2)  ^{-1} \ri).
 $$
 Then $\psi(\eta)$ is  a smooth function of $X,\by$. Since $   \psi '(0)=1$, we conclude that by implicit function theorem, $\eta$ is a smooth function in $X, \by$. From this it is easy to see the claim is true.
\end{proof}

Consider a spacelike CMC surface with constant mean curvature $H$ with respect to the future directed unit normal in Region I, which is the graph of a function  $u=u(\eta,\mathbf{y})$. We also assume that Assumption \ref{assumption1} is true. By the previous lemma, $u$ is smooth up to $\eta$. We want to consider the behaviors of $u$ near $\eta=0$. We have the following:

\begin{thm}\label{t-blackhole} All derivatives of
$u$ at $\eta=0$ can be    expressed in terms of the derivatives of $u$ with respect to $\mathbb{S}^2$. Moreover, at $\eta=0$,
$$
\p_\eta u=24m^2H.
$$

\end{thm}
\begin{proof} By Lemma \ref{l-equation-eta}, we have
\bee
\begin{split}
3 H L^\frac32=&  m^2r^{-4} \lf(L u_{\eta\eta}+(L\eta^{-1}-\frac12L_\eta)  u_\eta \ri)  +r^{-2} \lf(L\wt\Delta u-\frac12 \la\wn L,\wn u\ra  \ri) \\
=:&A+B
\end{split}
\eee
where
$$
L= \eta^{-2}-m^2r^{-4}u_\eta^2-r^{-2}|\wn u|^2,
$$
here $ \wn u  $ and $|\wn u|$ are taken with respect to the standard metric of $\mS^2$ etc.
We want to expand the terms as a power of $\eta$. By
\be\label{e-RHS-1}
\begin{split}
 \eta^3\times A
=&m^2r^{-4}\lf(\eta-\lf(m^2r^{-4}u_\eta^2  + r^{-2}|\wn u|^2\ri)\eta^3\ri)u_{\eta\eta} \\
&+m^2r^{-4}\lf(1-\lf(m^2r^{-4}u_\eta^2 + r^{-2}|\wn u|^2\ri)\eta^2\ri)u_\eta\\
&-\frac12m^2r^{-4}\lf( -2-\lf(2m^2r^{-4}u_\eta u_{\eta\eta}+r^{-2}(|\wn u|^2)_\eta\ri) \eta^3\ri)u_\eta\\
&-\frac12m^2r^{-4}
\lf(4m r^{-3}u_\eta^2 +2r^{-3}m^{-1}  r^2|\wn u|^2\ri)\eta^4 u_\eta.
\end{split}
\ee
\be\label{e-RHS-2}
\begin{split}
 \eta^3\times B
=&r^{-2}\lf(\eta-\lf(m^2r^{-4}u_\eta^2  + r^{-2}|\wn u|^2\ri)\eta^3\ri)\wt\Delta u
\\
&+\frac12 r^{-2}\eta^3\la    (m^2r^{-4}\wn u_\eta^2  + r^{-2}\wn (|\wn u|^2) ,\wn u\ra
\end{split}
\ee
then
\be\label{e-RHS-3}
\eta^3\times\mathbf{RHS}=\sum_{k=0}^4a_k\eta^k
\ee
where
\bee
\left\{
  \begin{array}{ll}
    a_0=2m^2r^{-4}u_\eta, \\
a_1=m^2r^{-4}u_{\eta\eta}+r^{-2}\wt \Delta u\\
a_2=-m^2r^{-4}\lf(m^2r^{-4}u_\eta^2 + r^{-2}|\wn u|^2\ri) u_\eta\\
a_3=-m^2r^{-4}\lf(m^2r^{-4}u_\eta^2  + r^{-2}|\wn u|^2\ri) u_{\eta\eta} - r^{-2}\wt\Delta u\lf(m^2r^{-4}u_\eta^2  + r^{-2}|\wn u|^2\ri)
\\
  + \frac12m^2r^{-4}\lf(2m^2r^{-4}u_\eta u_{\eta\eta}+r^{-2}(|\wn u|^2)\ri)u_\eta +\frac12 r^{-2} \la    (m^2r^{-4}\wn u_\eta^2  + r^{-2}\wn (|\wn u|^2) ,\wn u\ra\\
a_4=- m^2r^{-4}
\lf(2m r^{-3}u_\eta^2 + r^{-3}m^{-1}  r^2|\wn u|^2\ri)  u_\eta
  \end{array}
\right.
\eee

On the other hand,
\be\label{e-LHS}
\begin{split}
\eta^3\times \mathbf{LHS}
=&3 H \lf(1-\eta^2\lf(m^2r^{-4}u_\eta^2  + r^{-2}|\wn u|^2\ri)\ri)^\frac32\\
=&3  H \lf(1+\sum_{k=1}^\infty b_k\eta^{2k}\ri)
\end{split}
\ee
where
\bee
b_k=(-1)^k\left(
            \begin{array}{c}
              \frac32 \\
              k \\
            \end{array}
          \right)
\lf(m^2r^{-4}u_\eta^2+r^{-2}|\wt\nabla u|^2\ri)^k
\eee
and \bee
\left(
            \begin{array}{c}
              \frac32 \\
              k \\
            \end{array}
          \right)=\frac1{k!}\cdot\frac32(\frac32-1)\dots(\frac32-k+1).
       \eee
In general, note that $r^{-1}=\frac1{2m}(1-h)=\frac1{2m}(1-\eta^2)$. In the following, let us denote
$S_k$ to be a polynomial of derivatives w.r.t. $\mathbb{S}^2$ and $\eta$ of $\p^i_\eta u$ up to order $k$, it may vary from line to line. Then we have
\bee
\eta^3\times\mathbf{RHS}=\frac1{8m^2}u_\eta+\frac1{16m^2}u_{\eta\eta}\eta+S_1\eta+S_2\eta^2.
\eee
Hence at $\eta=0$, for $k\ge 1$,
\bee
\p^k_\eta\lf(\eta^3\times\mathbf{RHS}\ri)=\lf(\frac1{8m^2}+\frac k{16m^2}\ri)\p^{k+1}_\eta u +S_k
\eee
Also
\bee
\p^k_\eta\lf(\eta^3\times\mathbf{LHS}\ri)=S_k.
\eee
Hence if $\p^i_\eta u$ can be expressed in terms of the boundary data when $\eta=0$ for $1\le i\le k$, then this is also true for  $\p^{k+1}_\eta u$. Let us compute $\p_\eta u$. Let $\eta=0$ in  \eqref{e-LHS} and \eqref{e-RHS-3}, we have at $\eta=0$.
$$
u_\eta=24m^2  H .
$$
\end{proof}

\begin{rem}
 We remark that from the theorem, consider two spacelike CMC surfaces with constant mean curvature $H$ given as graphs of $u_1, u_2$ in $r>2m$. Suppose $u_1, u_2$  satisfy the conditions in the theorem. If $u_1(2m)=u_2(2m)$, then $u_1, u_2$ will agree up to infinite order at $\eta=0$ as functions of $\eta, \by$.
 \end{rem}

\subsection{Behaviors of general CMC surfaces near $T, X$=0 (II)}\label{ss-blackhole-cmc-2}
To investigate the totally geodesic of the inner boundary of a spacelike CMC hypersurface smooth near $T, X$=0, we adopt the original coordinates $(t,r, \theta,\phi)$ in Schwarzschild spacetime $\mathbb{X}^{3,1}$. Let $\Sigma$ be a spacelike  graph of function $t=u(r, \theta,\phi)$ in Region I, then the induced metric of $g_{Sch}$ on $\Sigma$ is
\begin{equation}\label{inducedmetric}
\begin{split}
 g|_\Sigma=&(h^{-1}-hu^2_r)dr^2+(r^2-hu^2_\theta)d\theta^2+(r^2\sin^2\theta-hu^2_\phi)d\phi^2\\
&-2hu_ru_\theta drd\theta -2hu_ru_\phi drd\phi-2h u_\theta u_\phi d\theta d\phi
\end{split}
\end{equation}

Denote $\p \Sigma:=\{(u(2m, \theta, \phi), 2m,\theta, \phi):(\theta, \phi)\in \mathbb{S}^2\}$, then $\p \Sigma$ is an inner boundary of $\Sigma$.
Let
\bee
e_1=h^\frac12(u_r \partial_{t} + \partial_r );
 e_2=u_\theta \partial_t + \partial_\theta ;
 e_3=u_\phi  \partial_t +  \partial_\phi,
\eee
then $\{e_i\}_{1\leq i\leq 3}$ are tangent vectors of $\Sigma$ and $\{e_2, e_3\}$ are tangent vectors of the inner boundary $\p \Sigma$.

\begin{thm}
Suppose a spacelike hypersurface $\Sigma$ satisfies Assumption \ref{assumption1},  and $u_\eta\neq 0$ at $r=2m$, then $\p \Sigma$ is totally geodesic in  $\Sigma$. In particularly, $\p \Sigma$ is totally geodesic for a spacelike CMC surface  satisfying Assumption \ref{assumption1}.
\end{thm}
\begin{proof}
Let
\bee
\rho(r)=\int^r_{2m}h^{-\frac12}(\tau)d\tau,
\eee
then $\partial_\rho=h^\frac12  \partial_r$ and for any $r>2m$,

 $$h^\frac12  u_r=u_\eta \frac{m}{r^2}.$$

By the conditions,  we know that $h^{\frac12}u_r$ is smooth up to $r=2m$, and $ A:=\lim_{r\rightarrow 2m}h^{\frac12}u_{r} \neq 0$.
We claim that at $r=2m$,
\be\label{partialrho}
   \la \nabla_{e_i}e_j, \partial_{\rho}\ra=0,~ ~ i,j=2,3.
\ee
In fact, denote $(x^{0},x^{1},x^{2},x^{3})=(t,r,\theta, \phi)$, then direct computations give
\bee
\Gamma_{00}^1=\frac{h h'}{2};
\Gamma_{11}^1= -\frac{h'}{2h};
\Gamma_{22}^1=-rh;
\Gamma_{33}^1=-rh\sin^{2}\theta ; \Gamma_{ab}^1=0, a\neq b, 0\leq a,b\leq 3.
 \eee
Hence, at $r=2m$,
\begin{align*}
&\la \nabla_{e_2}e_2,  \partial_{\rho} \ra
\\=&h^{\frac{1}{2}}\left(u_\theta^{2}\la \nabla_{\partial _t}\partial _t, \partial_{r}\ra+ u_\theta \la \nabla_{\partial _\theta}\partial _t, \partial_{r}\ra+ u_\theta \la \nabla_{\partial _t}\partial _\theta, \partial_{r}\ra+  \la \nabla_{\partial _\theta}\partial _\theta, \partial_{r}\ra\right)
\\=&h^{\frac{1}{2}} (u_\theta^{2}  \frac{h h'}{2} -rh ) = u_\theta^{2}mr^{-2}h^{\frac{3}{2}} -rh^{\frac{3}{2}}
=0 ,
\end{align*}
\begin{align*}
&\la \nabla_{e_2}e_3,  \partial_{\rho} \ra
\\=&h^{\frac{1}{2}}\left(u_\theta u_\phi \la \nabla_{\partial _t}\partial _t, \partial_{r}\ra+ u_\phi \la \nabla_{\partial _\theta}\partial _t, \partial_{r}\ra+ u_\theta \la \nabla_{\partial _t}\partial _\phi, \partial_{r}\ra+  \la \nabla_{\partial _\theta}\partial _\phi, \partial_{r}\ra\right)
\\=&h^{\frac{1}{2}}(u_\theta u_\phi \frac{h h'}{2}) = u_\theta u_\phi mr^{-2}h^{\frac{3}{2}}
 = 0,
\end{align*}
\begin{align*}
&\la \nabla_{e_3}e_3,  \partial_{\rho} \ra
\\=&h^{\frac{1}{2}}\left(u_\phi^{2}\la \nabla_{\partial _t}\partial _t, \partial_{r}\ra+ u_\phi \la \nabla_{\partial _\phi}\partial _t, \partial_{r}\ra+ u_\phi \la \nabla_{\partial _t}\partial _\phi, \partial_{r}\ra+  \la \nabla_{\partial _\phi}\partial _\phi, \partial_{r}\ra\right)
\\=&h^{\frac{1}{2}}(u_\phi^{2}  \frac{h h'}{2} -rh\sin^{2}\theta) = u_\phi^{2}mr^{-2}h^{\frac{3}{2}} -r\sin^{2}\theta h^{\frac{3}{2}}
 = 0 .
\end{align*}


Note that the vector field $e_1$ is smooth in the hypersurface $\Sigma$, let
\be\label{v1}
v_1=\lim_{r\rightarrow 2m}e_1=A\partial_t + \partial_\rho ,
\ee
then by the assumption, we see that $v_1$ is normal vector of $\p \Sigma$ in $\Sigma$. On the other hand, the level surface of $t$ in $ \mathbb{X}^{3,1}$ is totally geodesic, hence, we have (when we do covariant derivatives, we always extend $u$ trivially in a neighborhood of $\p \Sigma$ in the hypersurface $\Sigma$)
\be\label{v2}
\la \nabla_{e_i}e_j, \partial_t \ra=0.
\ee
Combining with \eqref{partialrho}, \eqref{v1} and \eqref{v2}, we get for $i,j=2,3$,
\bee
\la \nabla_{e_i}e_j,v_1\ra=0
\eee
which implies $\p \Sigma$ is totally geodesic in $\Sigma$. Together with Theorem\ref{t-blackhole}, we know the last conclusion is true.
\end{proof}

\end{document}